\documentclass{article}

\usepackage{amsmath,amssymb,amsthm}
\usepackage{mathtools}
\usepackage{bm}
\usepackage{mathrsfs}
\usepackage[all]{xy}
\usepackage{booktabs}
\usepackage{fullpage}
\usepackage{hyperref}
\usepackage{graphicx}
\hypersetup{
colorlinks=false,
linkcolor=red,
citecolor=green
}

\newcommand{\itg}{\mathbb{Z}}
\newcommand{\rtn}{\mathbb{Q}}
\newcommand{\rl}{\mathbb{R}}
\newcommand{\cx}{\mathbb{C}}

\newcommand{\ai}{\sqrt{-1}}

\newcommand{\inj}{\hookrightarrow}
\newcommand{\surj}{\twoheadrightarrow}
\newcommand{\isom}{\stackrel{\sim}{\to}}

\newcommand{\cst}{\textup{const}}
\newcommand{\kah}{K\"ahler }
\newcommand{\ke}{K\"ahler--Einstein }
\newcommand{\bl}[2]{\textup{Bl}_{#1} {#2}}

\newcommand{\ddbar}{\partial \bar{\partial}}

\newcommand{\actson}{\curvearrowright}
\newcommand{\prj}{\mathbb{P}}

\newcommand{\ocal}{\mathcal{O}}

\theoremstyle{plain}
\newtheorem{theorem}{Theorem}[section]
\newtheorem{lemma}[theorem]{Lemma}
\newtheorem{proposition}[theorem]{Proposition}

\theoremstyle{definition}
\newtheorem{definition}[theorem]{Definition}

\newtheorem{problem}[theorem]{Problem}

\newtheorem{notation}[theorem]{Notation}
\theoremstyle{definition}
\newtheorem{remark}[theorem]{Remark}
\newtheorem{question}[theorem]{Question}

\begin{document}

\title{Stability and canonical metrics on projective spaces blown up along a line}
\author{Yoshinori Hashimoto}

\maketitle

\abstract{Let $\text{Bl}_{\mathbb{P}^1} \mathbb{P}^n$ be a K\"ahler manifold obtained by blowing up a complex projective space $\mathbb{P}^n$ along a line $\mathbb{P}^1$. We prove that $\text{Bl}_{\mathbb{P}^1} \mathbb{P}^n$ does not admit constant scalar curvature K\"ahler metrics in any rational K\"ahler class, but admits extremal metrics, with an explicit formula in action-angle coordinates, in K\"ahler classes that are close to the pullback of the Fubini--Study class.}

\tableofcontents

\section{Introduction}

\subsection{Statement of the results}

Let $X$ be a compact \kah manifold. The existence of \kah metrics on $X$ with constant scalar curvature, or more generally extremal metrics as proposed by Calabi \cite{cal1}, is a question that has been intensively studied in the last few decades. Recall that a \kah metric $\omega$ is called \textbf{constant scalar curvature K\"ahler} (abbreviated as \textbf{cscK}) if its scalar curvature $S(\omega)$ is constant, and \textbf{extremal} if it satisfies $\bar{\partial} \mathrm{grad}^{1,0}_{\omega} S(\omega)=0$; the $(1,0)$-part of the gradient of its scalar curvature is a holomorphic vector field. Since $S(\omega)$ is defined to be the $\omega$-metric contraction of the Ricci curvature $\mathrm{Ric} (\omega) := - \ai \ddbar \log \det (\omega)$, obtaining cscK or extremal metrics amounts to solving a fully nonlinear partial differential equation (PDE), which is in general difficult. 

Consider now the following problem\footnote{This is mentioned, for example, in Sz\'ekelyhidi's survey \cite{szesurv} in the case $\dim_{\cx} Y >0$.}. 

\begin{problem} \label{blupprbhdmfd}
Suppose that a \kah manifold $X$ admits a cscK (resp. extremal) metric. Under what geometric hypotheses does the blowup $\mathrm{Bl}_Y X$ of $X$ along a complex submanifold $Y$ admit a cscK (resp. extremal) metric?
\end{problem}

The case $\dim_{\cx}  Y =0$ was solved by the theorems of Arezzo--Pacard \cite{ap1, ap2}, and Arezzo--Pacard--Singer \cite{aps}, which will be discussed in detail in \S \ref{blupcsckpts}, and will provide a background and motivation for considering Problem \ref{blupprbhdmfd}. The remaining case is $\dim_{\cx}  Y >0$, and we assume $\dim_{\cx} X \ge 3$ for the blowup to be non-trivial. On the other hand, there seems to be few results known about Problem \ref{blupprbhdmfd} when $\dim_{\cx}  Y >0$, and the solution of Problem \ref{blupprbhdmfd} in general seems to be very difficult at the moment; see \S \ref{prev} for the review of previously known results. 

We thus decide to focus instead on a particular example, the blowup $\textup{Bl}_{\mathbb{P}^1} \mathbb{P}^n$ of $\prj^n$ along a line, in the hope that this may serve as a useful example in attacking Problem \ref{blupprbhdmfd}. The result that we prove is the following.




\begin{theorem} \label{intstmresp1pn}
Let $n \ge 3$ and consider the blowup $\pi : \mathrm{Bl}_{\prj^1} \prj^n \to \prj^n$ of $\prj^n$ along a line $\prj^1$. $\mathrm{Bl}_{\prj^1} \prj^n$ is slope unstable (and hence $K$-unstable, cf.~\S \ref{slo}) with respect to any polarisation; in particular, $\mathrm{Bl}_{\prj^1} \prj^n$ cannot admit a cscK metric in any rational \kah class (cf.~\S \ref{ssecp1pnnonextres}). However, if we choose $\epsilon >0$ sufficiently small, there exists an extremal metric in the \kah class $\pi^* c_1 (\mathcal{O}_{\prj^n} (1)) - \epsilon c_1( [E])$, with an explicit formula given in Proposition \ref{mainthmextp1pn}, where $[E]$ is the line bundle associated to the exceptional divisor $E$.
\end{theorem}

\begin{notation} \label{notadtnpdnfisp}
In this paper, given a divisor $D$ in a \kah manifold $X$, we write $[D]$ for the line bundle $\mathcal{O}_{X}(D)$ associated to $D$. Also, we shall use the additive notation for the tensor product of line bundles, and the multiplicative notation will be reserved for the intersection product of divisors: given $n$ divisors $D_1 , \dots , D_n$, we shall write $D_1 . D_2 . \cdots . D_n$ to mean $\int_X c_1 ([D_1]) c_1 ([D_2]) \dots c_1 ([D_n])$, and $D_1^i . D_2^{n-i}$ to mean $\int_X c_1 ([D_1])^i c_1 ([D_2])^{n-i}$.
\end{notation}

\begin{remark}
In spite of its apparent simplicity, there has been no known result on $\mathrm{Bl}_{\prj^1} \prj^n$ in terms of cscK or extremal metrics, to the best of the author's knowledge (cf. \S \ref{prev}). This is perhaps related to the fact that $\mathrm{Bl}_{\prj^1} \prj^n$ does not admit a structure of a $\prj^1$-bundle; see \S \ref{calwotpbarr} for details. 


\end{remark}

\subsection{Blowup of cscK and extremal manifolds at points}  \label{blupcsckpts}
We now discuss the background for Theorem \ref{intstmresp1pn}, namely Problem \ref{blupprbhdmfd} for the case $\dim_{\cx}  Y=0$. We prepare some notation before doing so; write $\textup{Ham} (\omega , g)$ for the group of Hamiltonian isometries of $g$, i.e.~isometries of $(X,g)$ which are also Hamiltonian diffeomorphisms of $\omega$ and let $\mathfrak{ham} $ be its Lie algebra. We observe that $\textup{Ham} (\omega , g)$ is a finite dimensional compact Lie group. This allows us to define a moment map $m : X \to \mathfrak{ham}^*$, which we may normalise so that $\int_X \langle m, v \rangle \omega^n / n! =0$ for all $v \in \mathfrak{ham}$ with $\langle , \rangle$ being the natural duality pairing between $\mathfrak{ham}$ and $\mathfrak{ham}^*$. If $X$ admits a cscK metric, a classical theorem of Matsushima \cite{matsushima} and Lichnerowicz \cite{lichnerowicz} states the following for the Lie algebra of the group $\mathrm{Aut}_0 (X,L)$ consisting of the elements in the group of automorphisms $\mathrm{Aut} (X)$ which lift to the automorphism of the total space of an ample line bundle $L$ on $X$. 


\begin{theorem} \emph{(\cite{lichnerowicz, matsushima}; see also \cite[Theorem 1]{lebsim}, \cite[Theorems 6.1 and 9.4]{kobayashi})} \label{thmlsautxlcxham}
Suppose that $X$ admits a cscK metric. Writing $\mathfrak{aut}(X,L)$ for $\mathrm{LieAut}_0(X,L)$, we have $\mathfrak{aut} (X,L) = \mathfrak{ham}^{\cx}$.
\end{theorem}

We now consider Problem \ref{blupprbhdmfd} for $\dim_{\cx}  Y=0$. Suppose that we have a polarised cscK manifold $(X,L)$ which we blow up at points $p_1 , \dots , p_l$. We ask if the blown-up manifold $ \textup{Bl}_{p_1 , \dots , p_l} X$ admits a cscK metric in a ``perturbed'' \kah class so that the size of the exceptional divisor is small. Solution to this problem is given by the following theorem of Arezzo and Pacard \cite{ap2}, which generalises their previous result in \cite{ap1}.

\begin{theorem} \label{apthmpts}
\emph{(Arezzo--Pacard \cite{ap2})}
Let $(X,L)$ be a polarised \kah manifold with a cscK metric $\omega \in c_1 (L)$. Let $p_1 , \dots , p_l$ be distinct points in $X$ and $a_1 , \dots , a_l$ be positive real numbers. Suppose that the following conditions are satisfied:
\begin{enumerate}
\item $m(p_1) , \dots , m(p_l)$ spans $\mathfrak{ham}^*$,
\item $\sum_{i=1}^l a_i^{n-1} m(p_i) =0 \in \mathfrak{ham}^*$.
\end{enumerate}
Then there exists $\epsilon_0 >0$, $c >0$, and $\theta >0$ such that, for all $0< \epsilon < \epsilon_0$ the blowup $\hat{X}:= \textup{Bl}_{p_1, \dots, p_l} X$ of $X$ with the blowdown map $\pi : \hat{X} \to X$ admits a cscK metric $\omega_{\epsilon}$ in the perturbed \kah class
\begin{equation*}
 \pi^* [\omega] - \epsilon \sum_{i=1}^l \tilde{a_i} c_1([E_i])
\end{equation*}
where $\tilde{a_i}$ depends only on $\epsilon$ and satisfies $|\tilde{a_i} - a_i| \le c \epsilon^{\theta}$ as $\epsilon \to 0$ and $E_i$ stands for the exceptional divisor corresponding to the blowup at $p_i$. Moreover, $\omega_{\epsilon} \to \omega$ in the $C^{\infty}$-norm as $\epsilon \to 0$, away from $p_1 , \dots , p_l$.
\end{theorem}


By Theorem \ref{thmlsautxlcxham}, all of these hypotheses are vacuous if we assume $\mathfrak{aut} (X,L) = 0$. However, in presence of nontrivial holomorphic vector fields on $X$, we cannot choose the number and positions of $p_1, \dots , p_l$ arbitrarily to get a cscK metric on $\hat{X}$ (cf.~Theorem \ref{rtsistbblptpn}). 


Theorem \ref{apthmpts} has many differential-geometric and algebro-geometric applications \cite{dvz, donext, shu, stokst}; we note in particular that it was used to construct an example of asymptotically Chow unstable cscK manifold \cite{dvz}, and also to prove the $K$-stability of cscK manifolds with discrete automorphism group \cite{stokst}; see also \S \ref{ssecp1pnnonextres} for related discussions.

Even though $\mathfrak{aut} (X, L) \neq 0$ (or more precisely $\mathfrak{ham} \neq 0$) imposes some restrictions on the applicability of Arezzo--Pacard theorem, there is still hope of finding an extremal metric under weaker hypotheses, and moreover, it is natural to expect a version of this theorem for extremal metrics. Such result was indeed proved by Arezzo, Pacard, and Singer \cite[Theorem 2.0.2]{aps}. Just as Theorem \ref{apthmpts} was used by Stoppa \cite{stokst} to prove the $K$-stability of cscK manifolds when $\mathfrak{aut} (X,L)=0$, this result was used by Stoppa and Sz\'ekelyhidi \cite{stosze} to prove the relative $K$-stability of \kah manifolds with an extremal metric. We finally note that  Sz\'ekelyhidi \cite{szebl, szebl2} later established a connection to the $K$-stability of the blowup $\hat{X}$ when $X$ admits an extremal metric.

\subsection{Comparison to previous results} \label{prev}
We now return to the case $X = \prj^n$ and $Y = \prj^1$, to consider $\mathrm{Bl}_{\prj^1} \prj^n$. Our results (Theorem \ref{intstmresp1pn}) have much in common with, or more precisely are modelled after, the ones for the blowup $\mathrm{Bl}_{\mathrm{pt}} \prj^n$ of $\prj^n$ at a point. We review some previously known results on $\mathrm{Bl}_{\mathrm{pt}} \prj^n$ and its generalisations, as well as several nonexistence results that seem to be particularly relevant to Problem \ref{blupprbhdmfd}.

\subsubsection{Calabi's work on projectivised bundles and related results} \label{calwotpbarr}

In a seminal paper, Calabi \cite{cal1} presented the first examples of \kah manifolds which admit a non-cscK extremal metric. More precisely, he proved the following theorem.

\begin{theorem}
\emph{(Calabi \cite{cal1})} \label{calextpclbpn}
The projective completion $\prj (\mathcal{O}_{\prj^{n-1}} (-m) \oplus \cx) \to \prj^{n-1}$ of line bundles $\mathcal{O}_{\prj^{n-1}} (-m) \to \prj^{n-1}$, for any $m,n \in \mathbb{N}$, admits an extremal metric in each \kah class.
\end{theorem}

We observe that $\prj (\mathcal{O}_{\prj^{n-1}} (-1) \oplus \cx)$ is simply the blowup $\mathrm{Bl}_{\mathrm{pt}} \prj^n$ of $\prj^n$ at a point; the above theorem thus implies that there exists an extremal metric in each \kah class on $\mathrm{Bl}_{\mathrm{pt}} \prj^n$, although Theorem \ref{rtsistbblptpn} due to Ross and Thomas shows that none of these extremal metrics can be cscK.

There are two important features of $\mathrm{Bl}_{\mathrm{pt}} \prj^n$ (or more generally $\prj (\mathcal{O}_{\prj^{n-1}} (-m) \oplus \cx)$) that can be used in the construction of extremal metrics; the $\prj^1$-bundle structure and the toric structure. We first focus on the $\prj^1$-bundle structure. Calabi's original proof exploited this structure, which was later generalised by many mathematicians to various situations. While the reader is referred to \cite[\S 4.5]{hs} for a historical survey, we wish to particularly mention the following case to which this theory applies: suppose that we blow up two skew planes $P_1 \cong \prj^k$ and $P_2 \cong \prj^{n-k-1}$ in $\prj^n$. Then $\textup{Bl}_{P_1 , P_2} \prj^{n}$ is isomorphic to the total space of the projectivised bundle $\prj (\mathcal{O} (1,-1) \oplus \cx)$ over an exceptional divisor $\prj^k \times \prj^{n-k-1}$, where $\mathcal{O} (1,-1) = p^*_1 \mathcal{O}_{\prj^k}(1) \otimes p_2^* \mathcal{O}_{\prj^{n-k-1}} (-1)$ and $p_1 : \prj^k \times \prj^{n-k-1} \to \prj^k$ and $p_2 : \prj^k \times \prj^{n-k-1} \to \prj^{n-k-1}$ are the obvious projections. Then, we see that the following theorem of Hwang \cite{hwang} immediately implies that $\textup{Bl}_{P_1 , P_2} \prj^{n}$ carries an extremal metric in each \kah class.\footnote{In fact, some of the above examples admit \ke metrics, as shown by Koiso and Sakane \cite{ks}, Mabuchi \cite{mab87}, and also Nadel \cite{nad}.}

\begin{theorem}
\emph{(Hwang \cite[Corollary 1.2 and Theorem 2]{hwang})} \label{thhwapcomplb}
Let $(M, \omega_M)$ be a product of \ke Fano manifolds $(M_i , \omega_i)$, each with second Betti number 1. Let $\mathcal{F} := \bigotimes_{i=1}^r p_i^* K_i^{\otimes l_i} $, where $p_i : M \surj M_i$ is the obvious projection, $K_i$ is the canonical bundle of $M_i$, and $l_i \in \rtn$ so that $K_i^{\otimes l_i} $ is a (genuine) line bundle. Then the projective completion $\prj (\mathcal{F} \oplus \cx)$ of $\mathcal{F}$ over $M$ admits an extremal metric in each \kah class. 
\end{theorem}



On the other hand, $\mathrm{Bl}_{\prj^1} \prj^n$ does not have a structure of a $\prj^1$-bundle, so the above theorems do not apply. Thus, we now focus on the toric structure of $\mathrm{Bl}_{\mathrm{pt}} \prj^n$. This was treated in \cite{abrrev} and \cite{raza}, which (amongst other results) re-established Calabi's theorem using toric methods. This is the approach that we follow for $\mathrm{Bl}_{\prj^1} \prj^n$, and will be discussed in greater detail in \S \ref{ovpfexmp1pn}.

\subsubsection{Nonexistence results} \label{ssecp1pnnonextres}

We mention several nonexistence results that seem to be particularly relevant to Problem \ref{blupprbhdmfd}. We mainly focus on the results related to the Donaldson--Tian--Yau programme \cite{dontoric, tian97, yauprob}, which relates the existence of cscK metrics in the first Chern class $c_1 (L)$ of an ample line bundle $L$ to notions of algebro-geometric stability of $(X,L)$. On the other hand, since a detailed exposition on this topic could be rather lengthy, we do \textit{not} discuss it in detail here and we only recall some relevant notions and results that we shall use later; the reader is referred to e.g.~\cite{dontoric, tian97} for more details. A \textbf{test configuration} for a polarised \kah manifold $(X,L)$, written $(\mathcal{X} , \mathcal{L})$, is a flat family $\pi : \mathcal{X} \to \cx$ over $\cx$ with an equivariant $\cx^*$-action lifting to the total space of a line bundle $\mathcal{L}$ such that $\pi^{-1} (1)$ is isomorphic to $(X,L^{\otimes r})$, and $r$ is called the exponent of the test configuration $(\mathcal{X} , \mathcal{L})$. We can define a rational number called the \textbf{Donaldson--Futaki invariant} $DF(\mathcal{X} , \mathcal{L})$ for each $(\mathcal{X} , \mathcal{L})$ as in \cite[\S 2.1]{dontoric}, and $(X,L)$ is said to be \textbf{$K$-semistable} if $DF(\mathcal{X} , \mathcal{L}) \ge 0$ for every test configuration.

A foundational result is the following.

\begin{theorem} \emph{(Donaldson \cite{donlb})} \label{donlbcsckkst}
$(X,L)$ is $K$-semistable if it admits a cscK metric in $c_1 (L)$.
\end{theorem}

A polarised \kah manifold $(X,L)$ that is not $K$-semistable is called \textbf{$K$-unstable}. Thus, Theorem \ref{donlbcsckkst} shows that we can prove the nonexistence of cscK metrics in $c_1 (L)$ by proving that $(X,L)$ is $K$-unstable. Note that we also have a version of Theorem \ref{donlbcsckkst} adapted to the extremal metrics (cf.~\cite[Theorem 1.4 or 2.3]{stosze}).


Proving $K$-instability is often possible by establishing a stronger statement, which is to prove \textit{slope instability} of $(X,L)$; the reader is referred to \S \ref{slo} for more details on this theory. Along this line, we recall the following result of Ross and Thomas. We follow their approach very closely in proving the slope instability of $\mathrm{Bl}_{\prj^1} \prj^n$ (cf.~Proposition \ref{mainthminstp1pn}).

\begin{theorem} \label{rtsistbblptpn}
\emph{(Ross--Thomas \cite[Examples 5.27, 5.35]{rt})}
$\mathrm{Bl}_{\mathrm{pt}} \prj^n$ is slope unstable with respect to any polarisation. In particular, it cannot admit a cscK metric in any rational \kah class.
\end{theorem}


On the other hand, in some cases it is still possible to show $K$-instability directly, without proving slope instability, as in the following theorem due to Della Vedova \cite{dv}. They can be regarded as an extension of Stoppa's results \cite{stobl,stokst} to blowing up higher dimensional submanifolds. By defining the notion of ``Chow stability'' for subschemes inside a general polarised \kah manifold (cf.~\cite[Definition 3.5]{dv}), he proved the following by showing the $K$-instability of the blowup.

\begin{theorem}
\emph{(Della Vedova \cite[Theorem 1.5]{dv})}
Let $(X,L)$ be a polarised \kah manifold with a cscK metric in $c_1 (L)$. Let $Z_1 , \dots , Z_s$ be pairwise disjoint submanifolds of codimension greater than two, and let $\pi: \hat{X} \to X$ be the blowup of $X$ along $Z_1 \cup \dots \cup Z_s$ with $E_j$ being the exceptional divisor over $Z_j$. Define the subscheme $Z$ by the ideal sheaf $\mathcal{I}_Z := \mathcal{I}^{m_1}_{Z_1} \cap \dots \cap \mathcal{I}^{m_s}_{Z_s}$ for $m_1 , \dots , m_s \in \mathbb{N}$.

If $Z \inj X$ is Chow unstable, then the class $\pi^* c_1 (L) - \epsilon \sum_{j=1}^s m_j c_1([E_j])$ contains no cscK metrics for $0< \epsilon \ll 1$.
\end{theorem}



Della Vedova also proved an analogous statement for the extremal metrics \cite[Theorem 1.7]{dv}, by defining ``relative Chow stability'' for subschemes inside a general polarised \kah manifold. See Examples 1.6 and 1.11 in \cite{dv} for explicit examples in which these results are used.

\begin{remark}
Recalling Theorem \ref{thmlsautxlcxham}, we now ask whether the automorphism group of $X=\mathrm{Bl}_{\prj^1} \prj^n$ is reductive. It is easy to see that the Lie algebra $\mathfrak{aut} (X)$ of $\mathrm{Aut} (X)$ is equal to the Lie subalgebra $\mathfrak{h}$ of $\mathfrak{sl} (n+1, \cx)$ consisting of matrices of the form $\begin{pmatrix}
A & B \\
0 &C
\end{pmatrix}$ where $A,B,C$ are matrices of size $2 \times 2$, $2 \times (n-1)$, $(n-1) \times (n-1)$, respectively. Note that $X$ is Fano, and $\mathfrak{aut} (X) = \mathfrak{aut} (X,  -K_X)$. Note also that $\mathfrak{aut} (X,  -K_X) \cong \mathfrak{aut} (X,  L)$ for any ample line bundle $L$, cf.~\cite{kobayashi, lebsim}.

It is easy to see that the centre of $\mathfrak{h}$ is trivial, and hence $\mathfrak{h}$ is reductive if and only if it is semisimple. In principle this can be checked e.g.~by Cartan's criterion using the Killing form, although in practice it may be a nontrivial task. We can still prove that $\mathfrak{h}$ is not semisimple, and hence nonreductive, as follows. Theorem \ref{intstmresp1pn} shows that we have a non-cscK extremal metric in the polarisation $L:= \pi^* \mathcal{O}_{\prj^n} (1) - \epsilon [E]$ if $\epsilon >0$ is sufficiently small. This means that the Futaki invariant (cf.~\cite{futaki}) evaluated against the extremal vector field is not zero \cite[Lemma 1]{lebsim}. However, since the Futaki invariant is a Lie algebra character \cite[Corollary 2.2]{futaki}, this means that $\mathfrak{h} \cong \mathfrak{aut} (X,  L)$ cannot be semisimple. We thus conclude that $\mathfrak{h}$ is not reductive.

\end{remark}

\section{Some technical backgrounds}

We briefly recall slope stability in \S \ref{slo}, and toric \kah geometry in \S \ref{tor}. The aim of these sections is to fix the notation and recall some key facts; the reader is referred to the literature cited in each section for more details.

\subsection{Slope stability} \label{slo}
For the details of what is discussed in the following, the reader is referred to the paper \cite{rt} by Ross and Thomas.

Let $(X,L)$ be a polarised \kah manifold. Then for $k \gg 1$,
\begin{equation*}
\dim H^0 (X,  L^{\otimes k} ) = a_0 k^n + a_1 k^{n-1} + O(k^{n-2}) .
\end{equation*}

Now let $Z$ be a subscheme of $X$. The \textbf{Seshadri constant} $\mathrm{Sesh} (Z)$ for $Z \subset X$ (with respect to $L$) can be defined as follows. Considering the blowup $\pi : \mathrm{Bl}_Z X \to X$ with the exceptional divisor $E$, we define
\begin{equation*}
\textup{Sesh} (Z)= \mathrm{Sesh} (Z,X,L) := \sup \{ c \mid \pi^* L- cE \text{ is ample on } \mathrm{Bl}_Z X \} .
\end{equation*}
Then, writing $\mathcal{I}_Z$ for the ideal sheaf defining $Z$, we compute
\begin{equation*}
\dim H^0 (X,  L^{\otimes k} /(  L^{\otimes k} \otimes \mathcal{I}_{Z}^{xk} ) )= \tilde{a}_0 (x) k^n + \tilde{a}_1 (x) k^{n-1} + O(k^{n-2}) 
\end{equation*}
for $k \gg 1$ and $x \in \rtn$ such that $kx \in \mathbb{N}$. It is well-known that $\tilde{a}_i (x)$ is a polynomial in $x$ of degree at most $n-i$, and hence can be extended as a continuous function on $\rl$ (cf.~\cite[\S 3]{rt}).

\begin{definition}
The \textbf{slope} of $(X,L)$ is defined by $\mu (X, L):= a_1 / a_0 $, and the \textbf{quotient slope} of $Z$ with respect to $c \in \rl$ is defined by
\begin{equation*}
\mu_c (\mathcal{O}_Z , L) := \frac{\int_0^c \left(  \tilde{a}_1 (x) + \tilde{a}_0 (x) / 2   \right) dx}{\int_0^c  \tilde{a}_0 (x) dx} .
\end{equation*}
\end{definition}

\begin{definition} \label{defslstssustqsl}
$(X,L)$ is said to be \textbf{slope semistable with respect to} $Z$ if $\mu (X,L) \le \mu_c (\mathcal{O}_Z , L)$ for all $c \in (0, \textup{Sesh} (Z)]$. $(X,L)$ is said to be \textbf{slope semistable} if it is slope semistable with respect to all subschemes $Z$ of $X$. $(X,L)$ is \textbf{slope unstable} if it is not slope semistable.
\end{definition}

We remark that, since $X$ is a manifold, the slope can be computed by the Hirzebruch--Riemann--Roch theorem as
\begin{equation}  \label{comprsslope}
\mu (X,L) = - \frac{n \int_X c_1 (K_X) c_1 (L)^{n-1}}{2 \int_X c_1 (L)^n }.
\end{equation}

The quotient slope can also be computed in terms of Chern classes when $\mathrm{Bl}_Z X$ is smooth, by noting $\pi_* \mathcal{O}_{\textup{Bl}_Z X} (-jE) = \mathcal{I}_Z^j$ for $\pi : \mathrm{Bl}_Z X \to X$ and $j \ge 0$ and again using Hirzebruch--Riemann--Roch. It takes a particularly neat form when $Z$ is a divisor in $X$.
\begin{theorem} \emph{(Ross--Thomas \cite[Theorem 5.2]{rt})} \label{sldiv}
Let $Z$ be a divisor in $X$. Then
\begin{equation}  \label{comprsqslope}
\mu_c (\mathcal{O}_Z, L) = \frac{n \left( L^{n-1}. Z - \sum_{j=1}^{n-1} \binom{n-1}{j} \frac{(-c)^j}{j+1} L^{n-1-j} .Z^j . (K_X  + Z) \right)}{2\sum_{j=1}^{n} \binom{n}{j} \frac{(-c)^j}{j+1} L^{n-j} .Z^j},
\end{equation}
where the dot stands for the intersection product (cf.~Notation \ref{notadtnpdnfisp}), by identifying line bundles with corresponding divisors.
\end{theorem}
A fundamental theorem of Ross and Thomas is the following.
\begin{theorem} \emph{(Ross--Thomas \cite[Theorem 4.2]{rt})} \label{rtksssssssz}
If $(X,L)$ is $K$-semistable, then it is slope semistable with respect to any smooth subscheme $Z$.
\end{theorem}

\begin{remark}
Slope stability is strictly weaker than $K$-stability; the blowup of $\prj^2$ at two distinct points with the anticanonical polarisation is $K$-unstable, and yet slope stable \cite[Example 7.6]{pr}.
\end{remark}



\subsection{Toric K\"ahler geometry} \label{tor}
In addition to the original papers cited below, we mention \cite{abrrev, dontor} and Chapters 27-29 of \cite{cds} as particularly useful reviews on the details of what is discussed in this section.

We first of all demand that the symplectic form $\omega$ on $X$ be fixed throughout in this section. Recall that an action of a group $G$ on a manifold $X$ is called \textbf{effective} if for each $g \in G$, $g \neq \mathrm{id}_G$, there exists $x \in X$ such that $g \cdot x \neq x$. We first define a toric symplectic manifold, by regarding a \kah manifold $(X , \omega)$ merely as a symplectic manifold.

\begin{definition}
A \textbf{toric symplectic manifold} is a symplectic manifold $(X , \omega)$ equipped with an effective Hamiltonian action of an $n$-torus $T^n := \rl^n / 2 \pi \itg^n$ with a corresponding moment map $m : X \to \rl^n$.
\end{definition}

\begin{remark}
Recall that a \textbf{moment map} for the action $T^n \actson X$ is a $T^n$-invariant map $m : X \to \mathrm{Lie} (T^n)^* \cong \rl^n$ such that $\iota (v) \omega = -d \langle m ,v \rangle$ for all $v \in \mathrm{Lie}(T^n)$. A $T^n$-action is called \textbf{Hamiltonian} if there exists a moment map for the action.
\end{remark}

A theorem due to Atiyah \cite{ati}, and Guillemin and Sternberg \cite{gs} states that the image of the moment map $m$ is the convex hull of the images of the fixed points of the Hamiltonian torus action. For a toric symplectic manifold, it is a particular type of convex polytope called a \textbf{Delzant polytope}. Delzant \cite{del} showed that we have a one-to-one correspondence between a Delzant polytope $\mathcal{P}$ and a toric symplectic manifold $T^n \actson (X , \omega)$; Delzant polytopes are \textit{complete invariants} of toric symplectic manifolds. This allows us to confuse a toric symplectic manifold with its associated Delzant polytope $\mathcal{P}$, which is often called the \textbf{moment polytope}.

It is well-known that on (the preimage inside $X$ of) the interior $\mathcal{P}^{\circ}$ of the moment polytope $\mathcal{P}$, the $T^n$-action is free and we have a coordinate chart $\left\{ (x, y)= (x_1 , \dots , x_n , y_1 , \dots y_n) \in \mathcal{P}^{\circ} \times  (\rl^n / 2 \pi \itg^n)  \right\}$, called \textbf{action-angle coordinates}, on $m^{-1} (\mathcal{P}^{\circ})$. Action coordinates $(x_1 , \dots , x_n)$ are also called \textbf{momentum coordinates}. In action-angle coordinates, the symplectic form can be written as $\omega = \sum_{j=1}^n dx_j \wedge d y_j$ and the moment map can be given by $m (x,y) = x$.

We now consider a complex structure on $X$ to endow $(X , \omega)$ with a \kah structure; the reader is referred, for example, to \cite[\S 3 ]{abrrev} or \cite[\S 2]{dontor} for more details. We first recall that the way we construct $T^n \actson (X , \omega)$ from $\mathcal{P}$ \cite{del} shows that any toric symplectic manifold automatically admits a $T^n$-invariant \textit{complex} structure compatible with $\omega$; a toric symplectic manifold is automatically a toric \textit{K\"ahler} manifold. Let $\mathcal{S}_n$ be the \textbf{Siegel upper half space} consisting of complex symmetric $n \times n$ matrices of the form $Z = R + \ai S$ where $R$ and $S$ are real symmetric matrices and $S$ is in addition assumed to be positive definite. It is known that $\mathcal{S}_n$ is isomorphic to $Sp (2n , \rl) / U(n)$, and that $\mathcal{S}_n$ bijectively corresponds to the set $\mathcal{J}(\rl^{2n}, \omega_{\textup{std}})$ of all complex structures on $\rl^{2n}$ which are compatible with its standard symplectic form $\omega_{\textup{std}}$. It follows that in the action-angle coordinates on $m^{-1} (\mathcal{P}^{\circ})$, by taking a Darboux chart, any almost complex structure $J$ on $(X, \omega)$ can be written as
\begin{equation*}
J = \begin{pmatrix}
- S^{-1} R & -S^{-1} \\ RS^{-1} R+S & RS^{-1}
\end{pmatrix} .
\end{equation*}
If we assume that $J$ is $T^n$-invariant, we can make $R$, $S$ depend only on the action coordinates $x$. Moreover, by a Hamiltonian action generated by a function $f(x)$, given infinitesimally as $y_j \mapsto y_j + \frac{\partial f}{\partial x_j} (x)$, we may choose $R=0$. Furthermore, if we choose $J$ to be integrable, we can show that there exists a potential function $s (x)$ of $S$ such that $S_{ij} = \frac{\partial^2 s}{ \partial x_i \partial x_j} (x)$. Such $s(x)$ is called a \textbf{symplectic potential}. Guillemin \cite{guil} showed that we can define a canonical complex structure, or canonical symplectic potential on $(X, \omega)$, from the data of the moment polytope $\mathcal{P}$. 

\begin{theorem} \emph{(Guillemin \cite{guil})} \label{guilbd}
Suppose that $\mathcal{P}$ has $d$ facets (i.e. codimension 1 faces) which are defined by the vanishing of affine functions $l_i : \rl^n \ni x  \mapsto l_i (x) := \langle x, \nu_i \rangle -\lambda_i \in \rl$, $i = 1, \dots , d$, where $\nu_i \in \itg^n$ is a primitive inward-pointing normal vector to the $i$-th facet and $\lambda_i \in \rl$. Then in the action-angle coordinates on $m^{-1} (\mathcal{P}^{\circ})$, the canonical symplectic potential $s_{\mathcal{P}} (x)$ is given by
\begin{equation*}
s_{\mathcal{P}} (x) := \frac{1}{2} \sum_{i=1}^d l_i (x) \log l_i (x) .
\end{equation*}
\end{theorem}

Note that $s_{\mathcal{P}} (x)$ is not smooth at the boundary of the polytope, and this singular behaviour will be important in what follows. Abreu \cite{abr03} further showed that \textit{all} $\omega$-compatible $T^n$-invariant complex structures can be obtained by adding a smooth function to the above $s_{\mathcal{P}} (x)$.

\begin{theorem} \emph{(Abreu \cite[Theorem 2.8]{abr03})} \label{abrcx}
An $\omega$-compatible $T^n$-invariant complex structure on a toric \kah manifold $(X , \omega)$ is determined by a symplectic potential of the form $s (x) := s_{\mathcal{P}} (x) + r(x)$ where $r (x)$ is a function which is smooth on the whole of $\mathcal{P}$ such that the Hessian $\mathrm{Hess}(s)$ of $s$ is positive definite on the interior of $\mathcal{P}$ and has determinant of the form
\begin{equation} \label{abrcxthmhesssdet}
\det \left( \mathrm{Hess}(s) (x) \right) = \left[ \delta(x) \prod_{i=1}^d l_i (x) \right]^{-1},
\end{equation}
with $\delta$ being a smooth and strictly positive function on the whole of $\mathcal{P}$. Conversely, any symplectic potential of this form defines an $\omega$-compatible $T^n$-invariant complex structure on a toric \kah manifold $(X , \omega)$.
\end{theorem}



The description in terms of the symplectic potential gives the scalar curvature a particularly neat form. Now let $g_s$ be the Riemannian metric defined by $\omega$ and the complex structure determined by the symplectic potential $s (x)$. Write $s^{ij} (x)$ for the inverse matrix of the Hessian $\frac{\partial^2 s}{ \partial x_i \partial x_j} (x)$. Abreu \cite{abr98} derived the following equation in the action-angle coordinates.

\begin{theorem} \label{abrscaacoords}
\emph{(Abreu \cite[Theorem 4.1]{abr98})} \label{abreq}
The scalar curvature $S(g_s)$ of $g_s$ can be written as
\begin{equation}
S (g_s) = - \frac{1}{2} \sum_{i,j = 1}^n \frac{\partial^2 s^{ij}}{ \partial x_i \partial x_j} (x). \label{abreq1}
\end{equation}
Moreover, $g_s$ is extremal if and only if
\begin{equation}
\frac{\partial}{\partial x_k} S (g_s) = \cst \label{abreq2}
\end{equation}
for all $k = 1, \dots , n$.
\end{theorem}

The equation (\ref{abreq1}) is often called \textbf{Abreu's equation}.

\section{Slope instability of $ \mathrm{Bl}_{\prj^1} \prj^n$}

\subsection{Statement of the result}

We now return to the case where we blow up a line $\prj^1$ inside $\prj^n$, where we assume $n \ge 3$ for the blowup to be nontrivial. For ease of notation, we write $X := \mathrm{Bl}_{\prj^1} \prj^n$ and also write $\pi$ for the blowdown map $\pi : X \to \prj^n$. We re-state the first part of Theorem \ref{intstmresp1pn} as follows. 

\begin{proposition} \label{mainthminstp1pn}
$X = \mathrm{Bl}_{\prj^1} \prj^n$, $n \ge 3$, is slope unstable with respect to any polarisation. In particular, $X$ cannot admit a cscK metric in any rational \kah class.
\end{proposition}


\subsection{Proof of Proposition \ref{mainthminstp1pn}}

\subsubsection{Preliminaries on intersection theory} \label{premointth}

Observe first of all that any line bundle $L$ on $X= \mathrm{Bl}_{\prj^1} \prj^n$ can be written as $L = a \pi^* \ocal_{\prj^n} (1) - b  [E]$, with some $a,b \in \mathbb{Z}$, by recalling $\mathrm{Pic} (X)= \itg \pi^*\ocal_{\prj^n} (1) \oplus \itg [E]$. This is ample if and only if $a >b>0$. Thus, up to an overall scaling, we may say that any ample line bundle on $X$ can be written, as a $\rtn$-line bundle, as $L = \pi^* \ocal_{\prj^n} (1) - \epsilon  [E]$ for some $\epsilon  \in \rtn \cap (0,1)$.

This also implies $\mathrm{Sesh} (E, X, L) = 1 - \epsilon$; suppose that we blow up $E$ in $X$, with the blowdown map $\tilde{\pi} : \mathrm{Bl}_E X \cong X \isom X$. Then $\tilde{\pi}^* L-c[E] = \pi^* \ocal_{\prj^n} (1) - (c+ \epsilon) [E]$, which is ample if and only if $ - \epsilon < c < 1 - \epsilon$.

Henceforth, to simplify the notation, we write $H$ for the hyperplane in $\prj^n$ so that $[H] = \ocal_{\prj^n} (1)$.


Our aim is to show that $X$ is slope unstable with respect to the exceptional divisor $E$. Since the slope (\ref{comprsslope}) and the quotient slope (\ref{comprsqslope}) can be computed in terms of intersection numbers, we first need to prepare some elementary results on the intersection theory on $X$; more specifically, we need to compute $ \int_X c_1 (\pi^* [H])^j c_1 ([E])^{n-j}$ for $0 \le j \le n$. 

Recall (e.g.~\cite[\S 3, Chapter 3]{grihar}) the Euler exact sequence 
\begin{equation*}
0 \to \mathcal{O}_{\prj^n} \stackrel{j}{\to} \mathcal{O}_{\prj^n}(1)^{\oplus (n+1)} \to T_{\prj^n} \to 0 ,
\end{equation*}
where the vector bundle homomorphism $j$ takes $1 \in \ocal_{\prj^n}$ to the Euler vector field
\begin{equation*}
 \sum_{i=0}^n Z_i \frac{\partial}{\partial Z_i} \in \bigoplus_{i=0}^n \left( \ocal_{\prj^n} (1) \frac{\partial}{\partial Z_i} \right) \cong \ocal_{\prj^n} (1)^{\oplus (n+1)} ,
\end{equation*}
with $[ Z_0 : \cdots : Z_n]$ being the homogeneous coordinates on $\prj^n$. Restricting this sequence to a line $\prj^1 \subset \prj^n$, we get $0 \to \mathcal{O}_{\prj^1} \stackrel{j}{\to} \mathcal{O}_{\prj^1}(1)^{\oplus (n+1)} \to T_{\prj^n} |_{\prj^1} \to 0$. Combining this with the exact sequences $0 \to T_{\prj^1} \to T_{\prj^n} |_{\prj^1} \to N_{\prj^1 / \prj^n}  \to 0$ and $0 \to \mathcal{O}_{\prj^1} \to \mathcal{O}_{\prj^1}(1)^{\oplus 2} \to T_{\prj^1} \to 0$, we get $N_{\prj^1 / \prj^n} \cong  \mathcal{O}_{\prj^1}(1)^{\oplus (n-1)}$. Thus the exceptional divisor $E = \prj (N_{\prj^1 / \prj^n})$ is isomorphic to $\prj ( \mathcal{O}_{\prj^1}(1)^{\oplus (n-1)} ) \cong \prj ( \mathcal{O}_{\prj^1}^{\oplus (n-1)}) \cong \prj^1 \times \prj^{n-2}$. Note also that the adjunction formula (e.g.~\cite[\S 1, Chapter 1]{grihar}) shows $[E] |_E \cong N_{E /X} $, and that $N_{E/X}$ is isomorphic to the tautological bundle $\ocal_E (-1)$ over $E = \prj ( \mathcal{O}_{\prj^1}(1)^{\oplus (n-1)} )$. We observe that $\ocal_E (-1) \cong  p_1^* \mathcal{O}_{\prj^1} (1) \otimes  p_2^* \mathcal{O}_{\prj^{n-2}} (-1)$, where $p_1$ (resp.~$p_2$) is the natural projection from $E$ to $\prj^1$ (resp.~$\prj^{n-2}$), and that $\pi^* [H] |_E \cong  p_1^* \mathcal{O}_{\prj^1} (1) \otimes p_2^* \mathcal{O}_{\prj^{n-2}}$. 




With these observations, and recalling that $c_1 ([E])$ is the Poincar\'e dual of $E$, we compute
\begin{align*}
E^n &= \int_X c_1 ([E])^n = \int_E c_1 ([E])^{n-1} \\
&= \int_{\prj^1 \times \prj^{n-2}} ( p_1^* c_1 (\mathcal{O}_{\prj^1} (1)) -   p_2^* c_1 ( \mathcal{O}_{\prj^{n-2}} (1)) )^{n-1} \\
&=(-1)^{n-2}(n-1)
\end{align*}
and
\begin{align*}
\pi^* H . E^{n-1} &= \int_X c_1 (\pi^* [H]) c_1 ([E])^{n-1} \\ 
&=\int_{\prj^1 \times \prj^{n-2}} p_1^* c_1 (\mathcal{O}_{\prj^1} (1)) ( p_1^* c_1 (\mathcal{O}_{\prj^1} (1)) -   p_2^* c_1 ( \mathcal{O}_{\prj^{n-2}} (1)) )^{n-2} \\
&= (-1)^{n-2} .
\end{align*}
If $2 \le j <n$, we have
\begin{align*}
\pi^* H^j . E^{n-j} 
&= \int_{\prj^1 \times \prj^{n-2}} p_1^* c_1 (\mathcal{O}_{\prj^1} (1))^j ( p_1^* c_1 (\mathcal{O}_{\prj^1} (1)) -   p_2^* c_1 ( \mathcal{O}_{\prj^{n-2}} (1)) )^{n-j-1} \\
&=0
\end{align*}
and
\begin{equation*}
\pi^* H^n = \int_X \pi^* c_1 ([H])^n =\int_{\pi (X)}  c_1 ([H])^n  = \int_{\prj^n} c_1 ([H])^n = 1 .
\end{equation*}

Summarising the above, we get the following lemma.

\begin{lemma} \label{intersecp1pn}
Writing $x := c_1 (\pi^* [H])$ and $y := c_1 ([E])$, we have the following rules:
\begin{enumerate}
\item $x^n =1$,
\item $x y^{n-1} = (-1)^{n-2}$,
\item $y^n = (-1)^{n-2} (n-1) $,
\item $x^j y^{n-j} = 0$ for $2 \le j \le n-1$.
\end{enumerate}
\end{lemma}

\subsubsection{Computation of the slope $\mu (X,L)$}
We apply Lemma \ref{intersecp1pn} to the formula (\ref{comprsslope}) for the slope $\mu (X,L)$. Recall first of all that we have $K_X = \pi^* K_{\prj^n} + (n-2) [E]$ since we have blown up a complex submanifold of codimension $n-1$ (cf.~\cite[\S 6, Chapter 4]{grihar}). Note that $L  = \pi^* [H] - \epsilon  [E]$ and $K_X = \pi^* K_{\prj^n} + (n-2) [E]$ implies $c_1 (L) = x - \epsilon y$ and $c_1 (K_X) = -(n+1)x + (n-2)y$. We thus get
\begin{equation*}
 \int_X c_1 (K_X) c_1 (L)^{n-1} 
 =-(n+1)(1-\epsilon^{n-1}) + (n-1)(n-2) \epsilon^{n-2} (1-\epsilon ) 
\end{equation*}
by Lemma \ref{intersecp1pn}. Similarly we get $\int_X c_1 (L)^n =1 - n\epsilon ^{n-1} +(n-1) \epsilon^n$. Hence
\begin{equation*}
\mu (X,L) = \frac{n}{2}\frac{(n+1)(1-\epsilon^{n-1}) -(n-1)(n-2) \epsilon^{n-2} (1-\epsilon ) }{ (n-1) \epsilon^n- n \epsilon^{n-1}+1} .
\end{equation*}
For the later use, we write $\mathrm{Den}_1$ for the denominator and $\mathrm{Num}_1$ for the numerator of the fraction above, so that $\mu (X,L) = \frac{n}{2} \mathrm{Num}_1 / \mathrm{Den}_1$.

\subsubsection{Computation of the quotient slope $\mu_c (\mathcal{O}_E , L)$}

We now compute the quotient slope $\mu_c (\mathcal{O}_E , L)$ with respect to the exceptional divisor $E$ and for $c = \textup{Sesh} (E,X,L) = 1-\epsilon $, by using the formula (\ref{comprsqslope}).


We write $ \mu_{\mathrm{Sesh}(E)} (\mathcal{O}_E, L) = \frac{n}{2} \mathrm{Num}_2 / \mathrm{Den}_2$ and compute the denominator $\mathrm{Den}_2$ and the numerator $\mathrm{Num}_2$ separately. We first compute the denominator by using Lemma \ref{intersecp1pn}.
\begin{align*}
\mathrm{Den}_2
&=\sum_{j=1}^{n} \binom{n}{j} \frac{(\epsilon -1)^j}{j+1} (x-\epsilon y)^{n-j} y^j \\
&= \sum_{j=1}^{n} \binom{n}{j} \frac{(1-\epsilon )^j}{j+1} \left( -(n-j) \epsilon^{n-j-1} + (n-1)\epsilon^{n-j} \right) \\
&=\epsilon ^{n-1} \left( \left(1+ \frac{1-\epsilon }{\epsilon }\right)^n -1  \right) +\sum_{j=1}^{n} \binom{n}{j} \epsilon ^{n-1}  \left(  - \frac{n+1}{j+1} \left(\frac{1-\epsilon }{\epsilon }\right)^j  + \frac{n-1}{j+1} \epsilon  \left(\frac{1-\epsilon }{\epsilon }\right)^j \right).
\end{align*}
We now set $\chi := \frac{1-\epsilon }{\epsilon }$ and note the following identity
\begin{equation} \label{elempolyidq}
\sum_{j=1}^{n} \binom{n}{j} \frac{\chi^j}{j+1} = \sum_{j=1}^{n} \binom{n}{j} \frac{1}{\chi} \int_0^{\chi}  T^j dT = \frac{1}{\chi } \left( \frac{(1+\chi )^{n+1} -1}{n+1} - \chi \right) . 
\end{equation}
Observing $1+\chi  = \epsilon ^{-1}$, we thus get the denominator as
\begin{equation*}
\mathrm{Den}_2 =  - \frac{1-\epsilon^n}{1-\epsilon }+ n \epsilon^{n-1}+  \frac{n-1}{n+1} \frac{1 -\epsilon^{n+1}}{1-\epsilon } -(n-1) \epsilon^n .
\end{equation*}


We now compute the numerator. Since the first term $L^{n-1}.E$ is equal to $(x-\epsilon y)^{n-1} y = (n-1) \epsilon^{n-2}(1-\epsilon )$, 
we are left to compute the following second term
\begin{align*}
&\sum_{j=1}^{n-1} \binom{n-1}{j} \frac{(\epsilon -1)^j}{j+1} L^{n-1-j} .E^j . ( K_X + E) \\
&=\sum_{j=1}^{n-1} \binom{n-1}{j} \frac{(\epsilon -1)^j}{j+1} (x-\epsilon y)^{n-1-j} y^j (-(n+1) x + (n-2)y +y).
\end{align*}
By applying Lemma \ref{intersecp1pn}, we can compute each summand as
\begin{align*}
(x-\epsilon y)^{n-1-j} &y^j (-(n+1) x + (n-2)y +y) \\
&= (-1)^j (n-1) (n-j-1) \epsilon^{n-j-2}  - (-1)^j \epsilon^{n-j-1} n(n-3)  .
\end{align*}
We thus get
\begin{align*}
&\sum_{j=1}^{n-1} \binom{n-1}{j} \frac{(\epsilon -1)^j}{j+1} L^{n-1-j} .E^j . ( K_X + E) \\
&=\sum_{j=1}^{n-1} \binom{n-1}{j} \frac{\epsilon^{n-2}}{j+1}  \left( (n-1) (n-j-1) \left( \frac{1-\epsilon }{\epsilon } \right)^j  - \epsilon  \left( \frac{1-\epsilon }{\epsilon } \right)^j n(n-3) \right) .
\end{align*}
Setting $\chi = \frac{1-\epsilon }{\epsilon }$ as we did before, the above is equal to
\begin{align*}
&\sum_{j=1}^{n-1} \binom{n-1}{j} \epsilon^{n-2} \left(  (n-1)^2 \frac{\chi^j}{j+1} -(n-1)j \frac{\chi^j}{j+1} - \epsilon  n(n-3) \frac{\chi^j}{j+1}  \right) \\
&= -(n-1) \epsilon^{n-2}\sum_{j=1}^{n-1} \binom{n-1}{j}  \chi^j +\epsilon^{n-2} \sum_{j=1}^{n-1} \binom{n-1}{j}  \left( n(n-1) \frac{\chi^j}{j+1}- \epsilon  n(n-3) \frac{\chi^j}{j+1}  \right) .
\end{align*}
Now recalling the identity (\ref{elempolyidq}), we see that the above is equal to
\begin{align*}
&-(n-1) \epsilon^{n-2} ((1+\chi )^{n-1} -1 ) + n(n-1) \epsilon^{n-2} \frac{1}{\chi} \left(\frac{(1+\chi )^n -1}{n} -\chi  \right) \\
&\ \ \ \ \ \ \ \ \ \ \ \ \ \ \ \ \ \ \ \  \ \ \ \ \ \ \ \ \ \ \ \ \ \ \ \ \ \ \ \ \ \ \ \ \ \ \ \ \ \ \  -n(n-3) \epsilon^{n-1} \frac{1}{\chi} \left(\frac{(1+\chi )^n -1}{n} -\chi  \right) \\
&= (n-1) \frac{1-\epsilon^{n-1}}{1-\epsilon } -(n-1)^2 \epsilon^{n-2}  -  (n-3) \frac{1 -\epsilon^n}{1-\epsilon } + n(n-3) \epsilon^{n-1} .
\end{align*}
Thus we find the numerator to be
\begin{equation*}
\mathrm{Num}_2 = (n-1) \epsilon^{n-2}(1-\epsilon ) - (n-1) \frac{1-\epsilon^{n-1}}{1-\epsilon } + (n-1)^2 \epsilon^{n-2}  + (n-3) \frac{1 -\epsilon^n}{1-\epsilon } - n(n-3) \epsilon^{n-1} .
\end{equation*}

\subsubsection{Proof of instability}
We now compute $\mu_{\mathrm{Sesh}(E)} (\mathcal{O}_E, L)- \mu (X,L)$. Since $\mu_{\mathrm{Sesh}(E)} (\mathcal{O}_E, L) - \mu (X,L) <0$ implies that $(X,L)$ is slope unstable with respect to the divisor $E$ (cf.~Definition \ref{defslstssustqsl}), it suffices to show that
\begin{equation*}
\frac{2}{n} (\mu_{\mathrm{Sesh}(E)} (\mathcal{O}_E, L) - \mu (X,L)) = \frac{\mathrm{Num}_2}{\mathrm{Den}_2} - \frac{\mathrm{Num}_1}{\mathrm{Den}_1} 
\end{equation*}
is strictly negative for all $0 < \epsilon  < 1$.

Since $\epsilon ^j > \epsilon ^{j+1}$ for any non-negative integer $j$ if $0 < \epsilon  < 1$, we have the following inequalities:
\begin{align}
\mathrm{Den}_2 &=- \frac{1-\epsilon^n}{1-\epsilon }+ n \epsilon^{n-1}+  \frac{n-1}{n+1} \frac{1 -\epsilon^{n+1}}{1-\epsilon } -(n-1) \epsilon^n \notag \\
&= - \frac{2}{n+1}\sum_{j=0}^{n-1} \epsilon^j +  \frac{n(1-n)}{n+1} \epsilon^n+ n \epsilon^{n-1}  \notag \\
&< - \frac{2n}{n+1} \epsilon^{n-1} +  \frac{n(1-n)}{n+1} \epsilon^{n-1}+ n \epsilon^{n-1} =0 \label{den2pos0q1}
 \end{align}
and
\begin{equation*} 
\mathrm{Den}_1 = 1 - n \epsilon^{n-1} +(n-1) \epsilon ^n = (1-\epsilon ) \left( -n \epsilon^{n-1} + \sum_{j=0}^{n-1} \epsilon^j \right) > 0,
\end{equation*}
for $0 < \epsilon  < 1$.

Thus, to show slope instability, we are reduced to proving $\mathrm{Num}_2 \mathrm{Den}_1 - \mathrm{Num}_1 \mathrm{Den}_2 >0$, or equivalently
\begin{equation*}
(1- \epsilon ) \left( \mathrm{Num}_2 \mathrm{Den}_1 - \mathrm{Num}_1 \mathrm{Den}_2 \right) > 0
\end{equation*}
for $0 < \epsilon <1$.


We first re-write $(1- \epsilon )  \mathrm{Num}_2$ as 
\begin{align*}
(1- \epsilon )  \mathrm{Num}_2 
&= (n-1) \epsilon^{n-2}(1-\epsilon )^2 -(n-1) (1- (n-1) \epsilon^{n-2} +(n-2) \epsilon^{n-1}) \\
&\ \ \ \ \ \ \ \ \ \ \ \ \ \ \ \ \ \ \ \ \ \ \ \ \ \ \ \ \ \ \ \ \ \ \ \ \ \ \ \ \ \ \ \ \ \ \ \ \ \ \ \ \ \ \ \ \ \ \ \ \ + (n-3) (1- n\epsilon^{n-1} +(n-1) \epsilon^n) .
\end{align*}
Let
\begin{equation*}
F_m := 1- m\epsilon ^{m-1} + (m-1) \epsilon ^m
\end{equation*}
be defined for an integer $m >1$. We record the following lemma which we shall use later.
\begin{lemma} \label{elempropforfm}
The following hold for $F_m$, where $m >1$ is an integer:
\begin{enumerate}
\item $F_m >  \frac{m(m-1)}{2} (1-\epsilon )^2 \epsilon^{m-2}>0$ for $0 < \epsilon <1$,
\item $F_m - F_{m-1} = (m-1)\epsilon^{m-2} (1-\epsilon )^2 >0$ for $0<\epsilon <1$,
\item $F_n = \mathrm{Den}_1$.
\end{enumerate}
\end{lemma}
\begin{proof}
Observe first of all 
\begin{align*}
F_m = 1-m\epsilon^{m-1} +(m-1)\epsilon ^m &=1-\epsilon^{m-1} - (m-1) \epsilon^{m-1} (1-\epsilon ) \\
&=(1-\epsilon ) \left( \sum_{j=0}^{m-2}(\epsilon^j - \epsilon^{m-1}) \right) \\
&=(1-\epsilon )^2 \left( \sum_{j=0}^{m-2} \left( \sum_{k=0}^{m-j-2}\epsilon^{j+k} \right) \right).
\end{align*}
Since $0<\epsilon <1$, we have $\epsilon^{j+1} <\epsilon ^j$ for any positive integer $j$. Thus
\begin{equation*}
F_m  > (1-\epsilon )^2 \left( \sum_{j=0}^{m-2} (m-j-1) \epsilon^{m-2} \right) = \frac{m(m-1)}{2}  (1-\epsilon )^2 \epsilon^{m-2} > 0,
\end{equation*}
proving the first item of the lemma. The second item follows from a straightforward computation. The third is a tautology.
\end{proof}

Using Lemma \ref{elempropforfm}, we can write $(1- \epsilon )  \mathrm{Num}_2 = (n-2)F_n -nF_{n-1}$, 
and hence 
\begin{equation*}
(1- \epsilon )  \mathrm{Num}_2  \mathrm{Den}_1 =  (n-2)F_n^2 -nF_{n-1} F_n.
\end{equation*}
Similarly, we compute $\mathrm{Num}_1  = (n-2)F_{n-1} + 3(1-\epsilon ^{n-1})$ 
and
\begin{equation}  \label{1qden2neg}
(1-\epsilon ) \mathrm{Den}_2 = - F_n + \frac{n-1}{n+1} F_{n+1}.
\end{equation}

Summarising these calculations, we finally get
\begin{align*}
&(1- \epsilon ) \left( \mathrm{Num}_2 \mathrm{Den}_1 - \mathrm{Num}_1 \mathrm{Den}_2 \right) \\
&= (n-2)F_n^2 -nF_{n-1} F_n +[(n-2)F_{n-1} + 3(1-\epsilon^{n-1})] \left( F_n - \frac{n-1}{n+1} F_{n+1} \right) ,
\end{align*}
and our aim now is to show that the right hand side of the above equation is strictly positive for all $0 < \epsilon < 1$.

By using Lemma \ref{elempropforfm}, we first re-write
\begin{align*}
(n-2)F_{n-1} + 3(1-\epsilon^{n-1}) &= (n-2)F_{n-1} + 3(F_n + (n-1) \epsilon^{n-1} (1-\epsilon )) \\
&= (n+1)F_n - (n-1)(n-2)\epsilon^{n-2}(1-\epsilon )^2 + 3(n-1) \epsilon ^{n-1} (1-\epsilon )
\end{align*}
so as to get
\begin{align*}
&(n-2)F_n^2 -nF_{n-1} F_n +[(n-2)F_{n-1} + 3(1-\epsilon^{n-1})] \left( F_n - \frac{n-1}{n+1} F_{n+1} \right)  \\
&= F_n \left(  -n F_{n-1}  +(2n-1) F_n - (n-1) F_{n+1} \right) \\
&\ \ \ \ \ \ \ \ \ \ \ \ \ \ \ \ \ \ \ \ \ + [ -(n-1)(n-2)\epsilon^{n-2}(1-\epsilon )^2 + 3(n-1) \epsilon^{n-1} (1-\epsilon )] \left( F_n - \frac{n-1}{n+1} F_{n+1} \right) .
\end{align*}
Now compute
\begin{align*}
-n F_{n-1} +(2n-1) F_n  - (n-1) F_{n+1} &= -n F_{n-1} +(n-2) F_n + (n+1)F_n - (n-1) F_{n+1} \\
&= n(n-1) \epsilon^{n-2} (1-\epsilon )^3,
\end{align*}
and get
\begin{align*}
&(1- \epsilon ) \left( \mathrm{Num}_2 \mathrm{Den}_1 - \mathrm{Num}_1 \mathrm{Den}_2 \right) \\
&=(n-1)\epsilon^{n-2} (1-\epsilon ) \left[ n(1-\epsilon )^2 F_n +((n+1)\epsilon -(n-2) ) \left( F_n - \frac{n-1}{n+1} F_{n+1} \right) \right] .
\end{align*}

Since $n(1-\epsilon )^2 F_n>0$ by Lemma \ref{elempropforfm} and
\begin{equation*}
  F_n - \frac{n-1}{n+1} F_{n+1}  = -(1-\epsilon )\mathrm{Den}_2> 0
\end{equation*}
by recalling (\ref{1qden2neg}) and (\ref{den2pos0q1}), we see that the above quantity is strictly positive if $(n+1)\epsilon -(n-2) \ge 0$, i.e. $\frac{n-2}{n+1} \le \epsilon  <1 $. This means that we have proved slope instability for $\frac{n-2}{n+1} \le \epsilon  <1 $.

Thus assume $ 0 < \epsilon  < \frac{n-2}{n+1}$ from now on. Now, again using Lemma \ref{elempropforfm}, we have
\begin{align*}
&(1- \epsilon ) \left( \mathrm{Num}_2 \mathrm{Den}_1 - \mathrm{Num}_1 \mathrm{Den}_2 \right) \\
&=(n-1)\epsilon^{n-2} (1-\epsilon ) \\
&\ \ \ \ \  \times  \left[ \left( n(1-\epsilon )^2 +2 \left( \epsilon -\frac{n-2}{n+1} \right) \right) F_n +[(n+1)\epsilon -(n-2) ] \left(- \frac{n(n-1)}{n+1} \epsilon^{n-1} (1-\epsilon )^2 \right) \right] .
\end{align*}
Noting $n(1-\epsilon )^2 +2\left( \epsilon -\frac{n-2}{n+1} \right) = n \left( \epsilon - \frac{n-1}{n} \right)^2 + \frac{5n-1}{n(n+1)}$ 
and also 
\begin{align*}
&[(n+1)\epsilon -(n-2) ] \left( - \frac{n(n-1)}{n+1} \epsilon^{n-1} (1-\epsilon )^2 \right) \\
&=n(n-1) \epsilon^{n-1}(1-\epsilon )^3 - \frac{3n(n-1)}{n+1} \epsilon^{n-1} (1-\epsilon )^2 ,
\end{align*}
we are thus reduced to proving that
\begin{equation*}
\left[ n \left( \epsilon - \frac{n-1}{n} \right)^2 + \frac{5n-1}{n(n+1)} \right] F_n  - \frac{3n(n-1)}{n+1} \epsilon^{n-1} (1-\epsilon )^2 +n(n-1) \epsilon^{n-1}(1-\epsilon )^3
\end{equation*}
is strictly positive for $ 0 < \epsilon  < \frac{n-2}{n+1}$.

Observe that $\frac{n-2}{n+1} < \frac{n-1}{n} $, which holds if $n \ge 1$, implies that $\left( \epsilon - \frac{n-1}{n} \right)^2$ is monotonically decreasing on $ 0 < \epsilon  < \frac{n-2}{n+1}$. Thus
\begin{equation*}
n \left( \epsilon - \frac{n-1}{n} \right)^2 + \frac{5n-1}{n(n+1)} > n \left( \frac{n-2}{n+1}- \frac{n-1}{n} \right)^2 + \frac{5n-1}{n(n+1)} =\frac{9n}{(n+1)^2} 
\end{equation*}
for  $ 0 < \epsilon  < \frac{n-2}{n+1}$. Hence, recalling Lemma \ref{elempropforfm}, we finally have
\begin{align*}
&\left[ n \left( \epsilon - \frac{n-1}{n} \right)^2 + \frac{5n-1}{n(n+1)} \right] F_n  - \frac{3n(n-1)}{n+1} \epsilon ^{n-1} (1-\epsilon )^2 \\
&>\frac{9n}{(n+1)^2} (1-\epsilon )^2 \epsilon^{n-2} \frac{n(n-1)}{2} - \frac{3n(n-1)}{n+1} \epsilon^{n-1} (1-\epsilon )^2 \\
&>(1-\epsilon )^2 \epsilon^{n-1} \frac{n(n-1)}{n+1} \left( \frac{9n}{2(n+1)}  - 3 \right) >0 
\end{align*}
for $ 0 < \epsilon  < \frac{n-2}{n+1}$, since $n \ge 3$. We have thus proved $(1- \epsilon ) \left( \mathrm{Num}_2 \mathrm{Den}_1 - \mathrm{Num}_1 \mathrm{Den}_2 \right) >0$ both for $ 0 < \epsilon  < \frac{n-2}{n+1}$ and $\frac{n-2}{n+1} \le \epsilon  <1$, finally establishing the slope instability for all $0 < \epsilon <1$.

\section{Extremal metrics on $\textup{Bl}_{\mathbb{P}^1} \mathbb{P}^n$}

\subsection{Statement of the result}

Having established the nonexistence of cscK metrics in Proposition \ref{mainthminstp1pn}, we now discuss the extremal metrics on $\textup{Bl}_{\mathbb{P}^1} \mathbb{P}^n$ ($n\ge 3$), with the blowdown map $\pi : \mathrm{Bl}_{\prj^1} \prj^n \to \prj^n$, as mentioned in the second part of Theorem \ref{intstmresp1pn}. We write $(x_1 , \dots , x_n)$ for the action coordinates on the moment polytope corresponding to $(\textup{Bl}_{\mathbb{P}^1} \mathbb{P}^n , \pi^* \mathcal{O}_{\prj^n} (1) - \epsilon [E])$, where the exceptional divisor is defined by $\{ \sum_{i=1}^{n-1} x_i = \epsilon \}$, and we write $r: = \sum_{i=1}^n x_i$ and $\rho := \sum_{i=1}^{n-1} x_i$; see \S \ref{ovpfexmp1pn} for more details. We re-state the second part of Theorem \ref{intstmresp1pn} as follows, with an explicit description of the extremal metrics in the action-angle coordinates.


\begin{proposition} \label{mainthmextp1pn}
There exists $0< \epsilon_0 <1$ such that $\mathrm{Bl}_{\prj^1} \prj^n$ admits an extremal \kah metric in the \kah class $\pi^*c_1 ( \mathcal{O}_{\prj^n} (1)) - \epsilon c_1 ([E])$ for any $\epsilon \in (0, \epsilon_0)$. Moreover, this metric admits an explicit description in terms of the symplectic potential $s(x)$ in the action-angle coordinates as follows:
\begin{equation}
s(x) = \frac{1}{2} \left( \sum_{i=1}^n x_i \log x_i + (1-r) \log (1-r) + h (\rho)  \right) \label{symppot}
\end{equation}
where $h(\rho)$ is given as an indefinite integral by
\begin{equation*}
h (\rho) =  \int^{\rho} d \rho   \int^{\rho} \frac{  -1  - \frac{2n+ \delta}{n(n-1)}   + \rho + \frac{(\delta - \gamma) \rho}{n(n+1)} +\frac{\gamma \rho^2}{(n+1)(n+2)}+ \alpha \rho^{-n}  + \beta \rho^{-n+1}     }{(1-\rho )   \left(    1- \rho  \left( 1  - \frac{2n+ \delta}{n(n-1)} + \frac{(\delta - \gamma) \rho }{n(n+1)} +\frac{\gamma \rho^2}{(n+1)(n+2)}+ \alpha \rho^{-n}  + \beta \rho^{-n+1} \right)  \right)} d \rho ,
\end{equation*}
with
\begin{align}
\alpha & = -1 - \frac{ \delta}{n(n+1)} - \frac{  \gamma}{(n+1)(n+2)} , \label{defofalphap1pn} \\
\beta &= \frac{n+1}{n-1} + \frac{ \delta}{n(n-1)} + \frac{ \gamma}{n(n+1)} , \label{defofbetap1pn} \\
\gamma &= \frac{n(n+1)(n+2)\left(  \left( \frac{ \epsilon^{n+1} -1}{n(n+1)} + \frac{\epsilon - \epsilon^n }{n(n-1)}  \right)\delta - 1 + \frac{n+1}{n-1} \epsilon - \epsilon^{n-1} + \frac{n-3}{n-1} \epsilon^n  \right)}{-n \epsilon^{n+2} + (n+2) \epsilon^{n+1} +n - (n+2) \epsilon}  , \label{defofgammap1pn} \\
\delta &= \left( \epsilon^{n-2} (1- \epsilon) - \frac{(-n \epsilon^{n+1} + (n+1) \epsilon^{n} -1)(n+2)\left( - 1 + \frac{n+1}{n-1} \epsilon - \epsilon^{n-1} + \frac{n-3}{n-1} \epsilon^n   \right)}{-n \epsilon^{n+2} + (n+2)  \epsilon^{n+1} +n - (n+2) \epsilon} \right. \notag \\
&\ \ \ \ \ \ \ \ \ \ \ \ \ \ \ \ \ \ \ \ \ \ \ \ \ \ \ \ \ \ \ \ \ \ \ \ \ \ \ \ \ \ \ \ \ \ \ \ \ \ \ \ \ \ \ \ \ \ \ \ \ \ \  \left. + \frac{n(n-3)}{n-1} \epsilon^{n-1} -(n-1) \epsilon^{n-2} + \frac{n+1}{n-1} \right) \notag \\
&\ \ \ \ \ \times \left( \frac{(-n \epsilon^{n+1} + (n+1) \epsilon^{n} -1)(n+2)}{-n \epsilon^{n+2} + (n+2) \epsilon^{n+1} +n - (n+2) \epsilon} \left( \frac{\epsilon^{n+1} -1}{n(n+1)} + \frac{\epsilon - \epsilon^n }{n(n-1)}  \right) \right. \notag \\
&\ \ \ \ \ \ \ \ \ \ \ \ \ \ \ \ \ \ \ \ \ \ \ \ \ \ \ \ \ \ \ \ \ \ \ \ \ \ \ \ \ \ \ \ \ \ \ \ \ \ \ \ \ \ \ \ \ \ \ \ \ \ \ \ \ \ \ \ \ \ \ \ \ \ \ \ \ \ \ \  \left.+ \frac{-(n-1) \epsilon^n +n \epsilon^{n-1} -1}{n(n-1)} \right)^{-1} . \label{defofdeltap1pn}
\end{align}
\end{proposition}

\begin{remark}
Note that the symplectic potential is well-defined up to affine functions, and hence the integration constants in $h(\rho)$ are not significant.
\end{remark}



\subsection{Proof of Proposition \ref{mainthmextp1pn}}

\subsubsection{Overview of the proof} \label{ovpfexmp1pn}
The basic strategy of the proof, as given in \S \ref{redexteq2ordlinode} and \S \ref{regsymppoth}, is exactly the same as in \cite[\S 5]{abrrev} or \cite[\S 4.2]{raza} for the point blow-up case; the crux of what is presented in the following is to show that the same strategy does indeed work for $\mathrm{Bl}_{\prj^1} \prj^n$, with an extra hypothesis $ \epsilon \ll 1$.


We recall that the moment polytope $\mathcal{P} (\prj^n)$ for $\prj^n$, with the Fubini--Study symplectic form, is the region in $\rl^n$ defined by the set of affine inequalities $\mathcal{P} (\prj^n) := \{ x_1 \ge 0, \dots , x_n\ge 0,  \sum_{i=1}^n x_i \le 1\}$ (cf.~Figure \ref{polytopep3}), where $(x_1, \dots , x_n) \in \rl^n $ are the action coordinates as defined in \S \ref{tor}.
\begin{figure}[p]
\begin{center}
\includegraphics[clip,width=8.0cm]{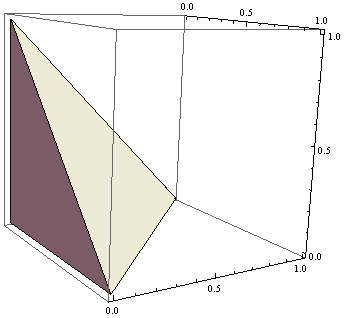}
\caption{The moment polytope $\mathcal{P} (\prj^3)$ for $\prj^3$.}
\label{polytopep3}
\end{center}
\end{figure}
The moment polytope $\mathcal{P}_{\epsilon} (X)$ for the blowup $X= \bl{\prj^1}{\prj^n}$ is obtained by cutting one edge by $\epsilon$ amount: $\mathcal{P}_{\epsilon} (X) := \{ x_1 \ge 0, \dots , x_n\ge 0,  \sum_{i=1}^n x_i \le 1 , \sum_{i=1}^{n-1} x_i \ge \epsilon \}$ (cf.~Figure \ref{polytopeblp1p3}), where the $\prj^1$ that is blown up corresponds to the line defined by $\{x_1 = \dots = x_{n-1} = 0 \}$. Note that the symplectic form $\omega$ on $X$ is in the cohomology class $\pi^* c_1 ( \mathcal{O}_{\prj^n} (1))  - \epsilon c_1 ([E])$ (cf.~\cite[Theorem 6.3]{guil}).
\begin{figure}[p]
\begin{center}
\includegraphics[clip,width=8.0cm]{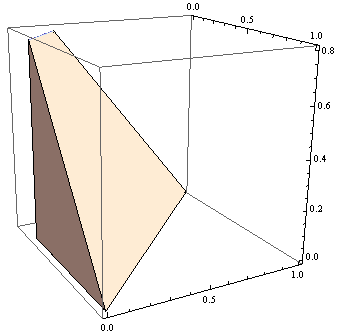}
\caption{The moment polytope $\mathcal{P}_{\epsilon} (X)$ for $X = \mathrm{Bl}_{\prj^1} \prj^3$, with $\epsilon = 0.2$.}
\label{polytopeblp1p3}
\end{center}
\end{figure}
We write $r := \sum_{i=1}^{n} x_i $ and $ \rho:= \sum_{i=1}^{n-1} x_i$ for notational convenience. Recall also that we assume $n \ge 3$ for the blow-up to be non-trivial.

Our strategy is to seek a symplectic potential $s$ of the form
\begin{equation} \label{symppotfsh}
s(x) = \frac{1}{2} \left( \sum_{i=1}^n x_i \log x_i + (1-r) \log (1-r) + h(\rho)   \right) ,
\end{equation}
where $h(\rho)$ stands for some function of $\rho$, so that the Riemannian metric $g_s$ given by the symplectic form $\omega$ and the complex structure defined by $s$ (cf.~Theorem \ref{abrcx}) satisfies the equation
\begin{equation} \label{exteqsprhoaac}
S (g_s) = - \gamma \rho - \delta 
\end{equation}
for some constants\footnote{The factor of $-1$ in (\ref{exteqsprhoaac}) is an artefact to be consistent with the equation (\ref{exteqaacredto2olode}).} $\gamma$ and $\delta$. Such a metric $g_s$ would be an extremal metric by Theorem \ref{abrscaacoords}. Our first result is that the equation (\ref{exteqsprhoaac}) reduces to a second-order linear ordinary differential equation (ODE) as given in (\ref{exteqaacredto2olode}), similarly to the case of the point blowup (cf.~\cite{abrrev, raza}). The equation (\ref{exteqaacredto2olode}) can be easily solved, and the solution is given in (\ref{solesteqaacrt2ode}) with two additional free constants $\alpha$ and $\beta$. This is the content of \S \ref{redexteq2ordlinode}.

However, it is not a priori obvious that $s(x) $ as defined in (\ref{symppotfsh}), with $h$ obtained from (\ref{solesteqaacrt2ode}), gives a well-defined symplectic potential. The main technical result (Proposition \ref{reghmaintr}) that we establish in \S \ref{regsymppoth} is that, once we choose $\alpha$, $\beta$, $\gamma$, $\delta$ as in (\ref{defofalphap1pn}), (\ref{defofbetap1pn}), (\ref{defofgammap1pn}), (\ref{defofdeltap1pn}) and $\epsilon$ to be sufficiently small, $h$ obtained from (\ref{solesteqaacrt2ode}) does satisfy all the regularity hypotheses required in Theorem \ref{abrcx}, so that $s(x)$ is a well-defined symplectic potential. This is the content of \S \ref{regsymppoth}.


\subsubsection{Reducing the equation (\ref{exteqsprhoaac}) to a second order linear ODE} \label{redexteq2ordlinode}

We first compute the Hessian
\begin{equation*}
s_{ij}: = \frac{\partial^2 s}{\partial x_i \partial x_j} (x)
\end{equation*}
of the symplectic potential $s(x) = \frac{1}{2} \left( \sum_{i=1}^n x_i \log x_i + (1-r) \log (1-r) + h(\rho)   \right)$ as follows:
\[ s_{ij} =
\begin{dcases}
& \frac{1}{2} \left( \frac{\delta_{ij}}{x_i} + \frac{1}{1-r} + h'' \right) \ \ \ \ \text{   if } i , j \neq n, \\
& \frac{1}{2} \left( \frac{\delta_{ij}}{x_i} + \frac{1}{1-r} \right) \ \ \ \ \ \ \ \ \ \ \ \ \text{   if } i =n \text{ or } j = n \text{ or both}.
\end{dcases}
\]
By direct computation, we find the inverse matrix $s^{ij}$ of $s_{ij}$ to be
\[ s^{ij} =
\begin{dcases}
& 2 \left( x_i \delta_{ij} -\frac{x_i x_j (1 + (1- \rho) h'' )}{1 +  \rho (1 - \rho) h''} \right) \ \ \ \ \ \ \ \ \ \text{   if } i , j \neq n, \\
& -\frac{ 2 x_i x_n }{1 + \rho (1 - \rho) h''} \ \ \ \ \ \ \ \ \ \ \ \ \ \ \ \ \ \ \ \ \ \ \ \ \ \ \ \ \ \ \text{   if }  i \neq n \text{ and } j=n, \\
 & \frac{ 2 x_n}{1- \rho} \left( 1-\rho -x_n + \frac{x_n \rho}{1 + \rho (1 - \rho) h''}    \right) \ \ \ \text{   if } i = j = n.
\end{dcases}
\]

Let $A$ be a function of $\rho$ defined by
\begin{equation*}
A (\rho) := \frac{1 + (1- \rho) h''}{1 +  \rho (1 - \rho)h''} ,
\end{equation*}
so that we can re-write the above as
\[ s^{ij} = 
\begin{dcases}
& 2(x_i \delta_{ij} - x_i x_j A) \ \ \ \ \ \ \ \ \ \ \ \ \ \ \ \ \ \ \ \ \  \text{   if } i , j \neq n, \\
& -\frac{ 2 x_i x_n (1- \rho A) }{1 - \rho} \ \ \ \ \ \ \ \ \ \ \ \ \ \ \ \ \ \  \text{   if } i \neq n \text{ and } j=n,\\
& \frac{ 2 x_n}{1- \rho} \left( 1-\rho -x_n + \frac{x_n \rho (1 - \rho A)}{1 - \rho}    \right) \ \ \ \text{   if } i = j = n.
\end{dcases}
\]
Thus, by Abreu's equation (\ref{abreq1}) (cf.~Theorem \ref{abrscaacoords}), we have
\begin{align*}
S (g_s) 
&= \sum_{i = 1}^{n-1} ( 2 A + 4 x_i A' + x_i^2 A'')+ 2 \sum_{1 \le i < j \le n-1 } (A + x_i A' + x_j A' + x_i x_j A'' ) \\
& \ \ \ \ \ + 2 \sum_{i = 1}^{n-1} \left(  \frac{1 - \rho A}{1 - \rho} + x_i \left ( \frac{1 - \rho A}{1 - \rho} \right)'    \right) + \left( \frac{2}{1- \rho} - \frac{2 \rho (1- \rho A)}{(1- \rho)^2}  \right) .
\end{align*}
Hence, re-arranging the terms, we find
\begin{equation*}
S (g_s) = \rho^2 A'' + 2 \left( n - \frac{\rho}{1 - \rho} \right) \rho A' + \left(  n(n-1) - \frac{2n \rho}{1- \rho}  \right) A + \frac{2n}{1- \rho} .
\end{equation*}
Thus the equation (\ref{exteqsprhoaac}) to be solved can now be written as
\begin{equation} \label{exteqaacredto2olode}
\rho^2 A'' + 2 \left( n - \frac{\rho}{1 - \rho} \right) \rho A' + \left(  n(n-1) - \frac{2n \rho}{1- \rho}  \right) A + \frac{2n}{1- \rho}+ \gamma \rho +\delta = 0
\end{equation}
for some constants $\gamma$ and $\delta$. The general solution to this equation is given by
\begin{equation*}
A= \frac{1}{1-\rho} \left( - \frac{2n+ \delta}{n(n-1)} + \frac{(\delta - \gamma) \rho}{n(n+1)} +\frac{\gamma \rho^2}{(n+1)(n+2)}+ \alpha \rho^{-n}  + \beta \rho^{-n+1} \right) ,
\end{equation*}
for some constants $\alpha$ and $\beta$. Recalling $A= \frac{1 + (1- \rho) h''}{1 +  \rho (1 - \rho) h''}$, we can now write $h''$ as
\begin{align}
h'' &= \frac{A-1}{(1- \rho)(1 - \rho A)} \notag \\
&=\frac{  -1  - \frac{2n+ \delta}{n(n-1)}   + \rho + \frac{(\delta - \gamma) \rho}{n(n+1)} +\frac{\gamma \rho^2}{(n+1)(n+2)}+ \alpha \rho^{-n}  + \beta \rho^{-n+1}     }{(1-\rho)   \left(    1- \rho \left( 1  - \frac{2n+ \delta}{n(n-1)} + \frac{(\delta - \gamma) \rho}{n(n+1)} +\frac{\gamma \rho^2}{(n+1)(n+2)}+ \alpha \rho^{-n}  + \beta \rho^{-n+1} \right)  \right)} . \label{solesteqaacrt2ode}
\end{align}

We have thus solved the equation (\ref{exteqsprhoaac}), with 4 undetermined parameters $\alpha$, $\beta$, $\gamma$, $\delta$. We now have to prove that the function $h$ as obtained above satisfies all the regularity conditions as stated in Theorem \ref{abrcx}, and we claim that this holds once $\alpha$, $\beta$, $\gamma$, $\delta$ are chosen as in (\ref{defofalphap1pn}), (\ref{defofbetap1pn}), (\ref{defofgammap1pn}), (\ref{defofdeltap1pn}).

Before discussing the claimed regularity of $h''$, which we do in \S \ref{regsymppoth}, we define two polynomials $P (\rho)$ and $Q(\rho)$, with $\alpha$, $\beta$, $\gamma$, $\delta$ as parameters, as follows. They play an important role in what follows.

\begin{definition}
We define a polynomial $P (\rho)$ by
\begin{equation*}
P(\rho) := - \frac{2n+ \delta}{n(n-1)}    + \frac{(\delta - \gamma) \rho}{n(n+1)} +\frac{\gamma \rho^2}{(n+1)(n+2)}
\end{equation*}
and $Q(\rho)$ by
\begin{align*}
Q (\rho )&:=\rho^{n-1}- \rho^n - \rho^{n}P(\rho) - \alpha  - \beta \rho \\
&= - \frac{\gamma }{(n+1)(n+2)} \rho^{n+2} - \frac{\delta - \gamma }{n(n+1)} \rho^{n+1} - \left( 1-\frac{2n + \delta}{n(n-1)} \right) \rho^n + \rho^{n-1} - \alpha - \beta \rho ,
\end{align*}
so that we can write
\begin{equation} \label{hdpitmsofpq}
h'' (\rho) = \frac{\rho^{n+1} - \rho^{n} + \rho^{n} P(\rho) + \alpha  + \beta \rho}{ (1- \rho) \rho Q (\rho) } .
\end{equation}
\end{definition}

\subsubsection{Regularity of $h$} \label{regsymppoth}

The main technical result is the following.

\begin{proposition} \label{reghmaintr}
For $h$ as given by (\ref{solesteqaacrt2ode}), there exists a function $R (\rho)$ which is smooth on the whole of the polytope $\mathcal{P}_{\epsilon} (X)$ such that
\begin{equation*}
h (\rho) = (\rho - \epsilon) \log (\rho - \epsilon) + R (\rho) 
\end{equation*}
and that the Hessian of the symplectic potential
\begin{equation*}
s(x) = \frac{1}{2} \left( \sum_{i=1}^n x_i \log x_i + (1-r) \log (1-r) + h (\rho)  \right)
\end{equation*}
is positive definite over the interior $\mathcal{P}^{\circ}_{\epsilon} (X)$ of the polytope $\mathcal{P}_{\epsilon} (X)$, with the determinant of the form required in (\ref{abrcxthmhesssdet}), if we choose $\alpha$, $\beta$, $\gamma$, $\delta$ as in (\ref{defofalphap1pn}), (\ref{defofbetap1pn}), (\ref{defofgammap1pn}), (\ref{defofdeltap1pn}) and $\epsilon >0$ to be sufficiently small.
\end{proposition}

\begin{proof}
Recall from (\ref{hdpitmsofpq}) that $h''$ is given by
\begin{equation*}
h'' (\rho) = \frac{\rho^{n+1} - \rho^{n} + \rho^{n} P(\rho) + \alpha  + \beta \rho}{ (1- \rho) \rho Q (\rho) } .
\end{equation*}

We first need to prove that $\rho =1$ is a removable singularity. In Lemma \ref{rho1remsingab}, we shall prove that this is indeed the case, once we choose $\alpha$ and $\beta$ as in (\ref{defofalphap1pn}), (\ref{defofbetap1pn}).

We then consider the asymptotic behaviour of $h''$ as $\rho \to \epsilon$. We now write
\begin{equation*}
h'' (\rho) = \frac{1}{\rho}  \frac{1}{\rho -1} + \frac{\rho^{n-2} (1- \rho)}{Q (\rho)}
\end{equation*}
and consider the Taylor expansion
\begin{equation*}
Q (\rho) = Q_{0} + Q_{1} (\rho - \epsilon) + \cdots
\end{equation*}
of $Q (\rho)$ around $\rho = \epsilon$, with some $Q_{0} , Q_{1} \in \rl$. Writing now
\begin{equation*}
h''(\rho) = \frac{1}{\rho}  \frac{1}{\rho -1} + \frac{\rho^{n-2} (1- \rho)}{Q_{0} + Q_{1} (\rho - \epsilon) + \cdots} 
\end{equation*}
around $\rho = \epsilon$, our strategy is to show that, for the choice of $\gamma$ and $\delta$ as in (\ref{defofgammap1pn}), (\ref{defofdeltap1pn}), we have a Laurent expansion
\begin{equation}
h'' (\rho)= \frac{1}{\rho - \epsilon} + \hat{Q}_{0} + \hat{Q}_{1} (\rho - \epsilon) + \cdots \label{expofhppatrhoeps}
\end{equation}
in $\rho - \epsilon$, with some $\hat{Q}_0 , \hat{Q}_1 \in \rl$. This will be proved in Lemma \ref{lempseofhdp}. We shall also prove in Lemma \ref{lemposqeps} that $Q (\rho) >0$ on $(\epsilon , 1)$, for these choices of $\alpha$, $\beta$, $\gamma$, $\delta$ and sufficiently small $\epsilon >0 $. Since $\rho =1$ is a removable singularity, this means that $h''$ is smooth on the whole polytope except for a pole of order 1 and residue 1 at $\rho = \epsilon$.

We now consider a function
\begin{equation*}
\tilde{R}(\rho) := h'' (\rho)-  \frac{1}{\rho - \epsilon} .
\end{equation*}
This is smooth on the whole of the polytope $\mathcal{P}_{\epsilon} (X)$ by the above properties of $h''$, and hence integrating both sides twice, we get a function $R(\rho)$ that is smooth on the whole polytope which satisfies
\begin{equation*}
R(\rho) = h (\rho) - (\rho - \epsilon) \log (\rho - \epsilon),
\end{equation*}
as we claimed. Finally, we shall prove in Lemma \ref{lemsijposdef} that the Hessian of the symplectic potential 
\begin{equation*}
s(x) = \frac{1}{2} \left( \sum_{i=1}^n x_i \log x_i + (1-r) \log (1-r) + h (\rho)  \right)
\end{equation*}
is indeed positive definite over the interior $\mathcal{P}^{\circ}_{\epsilon} (X)$ of the polytope $\mathcal{P}_{\epsilon} (X)$ and has determinant of the form required in (\ref{abrcxthmhesssdet}), for the above choices of $\alpha$, $\beta$, $\gamma$, $\delta$ and sufficiently small $\epsilon >0 $.

Therefore, granted Lemmas \ref{rho1remsingab}, \ref{lempseofhdp}, \ref{lemposqeps}, and \ref{lemsijposdef} to be proved below, we complete the proof of the proposition.

\end{proof}

\begin{lemma} \label{rho1remsingab}
For the choice of $\alpha$ and $\beta$ as in (\ref{defofalphap1pn}), (\ref{defofbetap1pn}), the following hold:
\begin{enumerate}
\item the numerator $\rho^{n+1} - \rho^{n} + \rho^{n} P(\rho) + \alpha  + \beta \rho$ of $h''$ has a zero of order at least 3 at $\rho =1$,
\item $Q(1) = Q'(1)=0$, $Q''(1) =2$; in particular, $Q$ has a zero of order exactly 2 at $\rho =1$.
\end{enumerate}
In particular, $\rho =1$ is a removable singularity of $h''$ if we choose $\alpha$ and $\beta$ as in (\ref{defofalphap1pn}), (\ref{defofbetap1pn}).
\end{lemma}

\begin{proof}
The numerator $\rho^{n+1} - \rho^{n} + \rho^{n} P(\rho) + \alpha  + \beta \rho$ of $h''$ has a zero at $\rho =1$ if and only if $P(1) + \alpha + \beta =0$. We thus choose $\beta = - \alpha -P(1)$. The zero is of order at least two if and only if $1+nP(1) + P'(1) + \beta = 0$, in addition to $\beta = - \alpha -P(1)$. We thus choose $\alpha$ by the equation
\begin{equation} \label{eqforalphap1pn}
1+nP(1) + P'(1) + (- \alpha -P(1)) = 0
\end{equation}
and $\beta$ by the equation
\begin{equation} \label{eqforbetap1pn}
\beta = - \alpha -P(1) = -1- (n-1) P(1) - P'(1) - P(1)
\end{equation}
by noting that $P(1)$ and $P'(1)$ depend only on $\gamma$ and $\delta$.

Finally, we observe
\begin{align}
\left. \frac{d^2}{d \rho^2} \right|_{\rho=1}  (\rho^{n+1} - \rho^{n} + \rho^{n} P(\rho) + \alpha  + \beta \rho) \notag 
&=  2n +n(n-1) P(1) + 2nP'(1) + P''(1) \notag \\
&=0  \label{zeropatrho13}
\end{align}
identically for any choice of $\gamma$ and $\delta$. Thus the numerator $\rho^{n+1} - \rho^{n} + \rho^{n} P(\rho) + \alpha  + \beta \rho$ of $h''$ vanishes at $\rho=1$ with order at least 3, if $\alpha$, $\beta$ are chosen as in the equations (\ref{eqforalphap1pn}), (\ref{eqforbetap1pn}). We now unravel the equations (\ref{eqforalphap1pn}) and (\ref{eqforbetap1pn}), to find that they are exactly as given in (\ref{defofalphap1pn}) and (\ref{defofbetap1pn}).


 We have thus established the first claim in the lemma: the numerator $\rho^{n+1} - \rho^{n} + \rho^{n} P(\rho) + \alpha  + \beta \rho$ of $h''$ has a zero of order at least 3 at $\rho =1$ if $\alpha$, $\beta$ are chosen as (\ref{defofalphap1pn}) and (\ref{defofbetap1pn}).

The second claim of the lemma is an easy consequence of the equations (\ref{eqforalphap1pn}), (\ref{eqforbetap1pn}), (\ref{zeropatrho13}): we simply compute $Q(1) = 1- 1- P(1) - \alpha - \beta =0$ and $Q'(1) = (n-1) -n -nP(1)-P'(1) -\beta =0$, by virtue of (\ref{eqforalphap1pn}) and (\ref{eqforbetap1pn}). We finally have $Q''(1) = 2$ by (\ref{zeropatrho13}).


\end{proof}

\begin{lemma} \label{lempseofhdp}
We have the expansion (\ref{expofhppatrhoeps}), namely we have the Laurent expansion
\begin{equation*}
h'' (\rho) = \frac{1}{\rho - \epsilon} + \hat{Q}_{0} + \hat{Q}_{1} (\rho - \epsilon) + \cdots 
\end{equation*}
in $\rho - \epsilon$, if we choose $\alpha$, $\beta$, $\gamma$, $\delta$ as in (\ref{defofalphap1pn}), (\ref{defofbetap1pn}), (\ref{defofgammap1pn}), (\ref{defofdeltap1pn}), and if $\epsilon$ is sufficiently small.
\end{lemma}

\begin{proof}
We first consider the Taylor expansion
\begin{equation} \label{texpofqreps}
Q (\rho) = Q_{0} + Q_{1} (\rho - \epsilon) + \cdots
\end{equation}
of $Q (\rho)$ around $\rho = \epsilon$, with $Q_{0} , Q_{1} \in \rl$. When we have $\alpha$ and $\beta$ as defined in (\ref{defofalphap1pn}) and (\ref{defofbetap1pn}), we find the 0th order term $Q_{0}$, which is equal to $Q(\epsilon)$, to be
\begin{align*}
Q_0 &=- \frac{\gamma }{(n+1)(n+2)} \epsilon^{n+2} - \frac{\delta - \gamma }{n(n+1)} \epsilon^{n+1} - \left( 1-\frac{2n + \delta}{n(n-1)} \right) \epsilon^n + \epsilon^{n-1} - \alpha - \beta \epsilon \\
&=\left( \frac{-n \epsilon^{n+2} + (n+2) \epsilon^{n+1} +n - (n+2) \epsilon}{n(n+1)(n+2)}  \right) \gamma \\
&\ \ \ \ \ \ + \left( \frac{1- \epsilon^{n+1} }{n(n+1)} + \frac{\epsilon^n - \epsilon}{n(n-1)}  \right) \delta  +1 - \frac{n+1}{n-1} \epsilon + \epsilon^{n-1} - \frac{n-3}{n-1} \epsilon^n.
\end{align*}
We choose $\gamma$ as in (\ref{defofgammap1pn}), 
so that $Q_0 =0$; note that $-n \epsilon^{n+2} + (n+2) \epsilon^{n+1} +n - (n+2) \epsilon \neq 0$ if $\epsilon$ is chosen to be sufficiently small. This means that we can write
\begin{align*}
h'' (\rho)&= \frac{1}{\rho}  \frac{1}{\rho -1} + \frac{\rho^{n-2} (1- \rho)}{Q(\rho)} \\
&= \frac{\rho^{n-2} (1- \rho)}{Q_1} \frac{1}{\rho - \epsilon} + \text{power series in } \rho - \epsilon,
\end{align*}
near $\rho = \epsilon$. In order to prove the stated claim, we need to show that the residue at the pole $\rho =\epsilon$ of $h''$ is 1. We prove this by showing $Q_1 = \epsilon^{n-2} (1 - \epsilon)$ for an appropriate choice of $\delta$, with $\alpha$, $\beta$, and $\gamma$ as determined in the above.

We thus consider the coefficient $Q_1$ in the expansion (\ref{texpofqreps}), which is equal to $\frac{d}{d \rho} \mid_{\rho = \epsilon} Q(\rho)$, i.e.
\begin{equation*}
Q_1=- \frac{\gamma }{n+1} \epsilon^{n+1} - \frac{\delta - \gamma }{n} \epsilon^{n} - \left( n-\frac{2n + \delta}{n-1} \right) \epsilon^{n-1}   +(n-1)\epsilon^{n-2} -\beta .
\end{equation*}
For the choice of $\beta$ and $\gamma$ as in (\ref{defofbetap1pn}) and (\ref{defofgammap1pn}), we can re-write this as
\begin{align}
Q_1= & \left[ \frac{(-n \epsilon^{n+1} + (n+1) \epsilon^{n} -1)(n+2)}{-n \epsilon^{n+2} + (n+2) \epsilon^{n+1} +n - (n+2) \epsilon} \left( \frac{\epsilon^{n+1} -1}{n(n+1)} + \frac{\epsilon - \epsilon^n }{n(n-1)}  \right) \right. \notag \\
& \ \ \ \ \ \ \ \ \ \ \ \ \ \ \ \ \ \ \ \ \ \ \ \ \ \ \ \ \ \ \ \ \ \ \ \ \ \ \ \ \ \ \ \ \ \ \ \ \ \ \ \ \ \ \ \ \ \ \ \ \ \ \ \ \ \ \ \ \ \ \ \ \ \left. + \frac{-(n-1) \epsilon^n +n \epsilon^{n-1} -1}{n(n-1)}    \right]                      \delta \notag \\
&\ \ \ \ \ \ \ \ \ \ \ \ \ + \frac{(-n \epsilon^{n+1} + (n+1) \epsilon^{n} -1)(n+2)}{-n \epsilon^{n+2} + (n+2)  \epsilon^{n+1} +n - (n+2) \epsilon} \left( - 1 + \frac{n+1}{n-1} \epsilon - \epsilon^{n-1} + \frac{n-3}{n-1} \epsilon^n   \right) \notag \\
& \ \ \ \ \ \ \ \ \ \ \ \ \ \ \ \ \ \ \ \ \ \ \ \ \ \ \ \ \ \ \ \ \ \ \ \ \ \ \ \ \ \ \ \ \ \ \ \ \ \ \ \ \ \ \ \ \ \ \ \ \ \ \ \ - \frac{n(n-3)}{n-1} \epsilon^{n-1} +(n-1) \epsilon^{n-2} - \frac{n+1}{n-1} . \label{eqdefofdeltaeps}
\end{align}

The equation $Q_1 = \epsilon^{n-2} (1 - \epsilon)$ can be solved for $\delta$ if and only if the coefficient of $\delta$ in the equation (\ref{eqdefofdeltaeps}) is not zero, i.e.
\begin{align*}
 &\frac{(-n \epsilon^{n+1} + (n+1) \epsilon^{n} -1)(n+2)}{-n \epsilon^{n+2} + (n+2) \epsilon^{n+1} +n - (n+2) \epsilon} \left( \frac{\epsilon^{n+1} -1}{n(n+1)} + \frac{\epsilon - \epsilon^n }{n(n-1)}  \right) + \frac{-(n-1) \epsilon^n +n \epsilon^{n-1} -1}{n(n-1)} \\
&\ \ \ \ \ \ \ \ \ \ \ \ \ \ \ \ \  \ \ \ \ \ \ \ \ \ \ \ \ \ \ \ \ \  \ \ \ \ \ \ \ \ \ \ \ \ \ \ \ \ \  \ \ \ \ \ \ \ \ \ \ \ \ \ \ \  \ \ \ \ \ \ \ \ \ \ \ \ \ \ \  \ \ \ \ \ \ \ \ \ \ \ \ \ \ \ \ \  \ \ \ \ \ \ \ \ \ \ \ \ \ \ \ \ \   \neq 0 .
\end{align*}
Note that the left hand side is equal to $\frac{-2}{n^2(n-1)(n+1)} \neq 0$ when $\epsilon = 0$, and hence this is non-zero for all sufficiently small $\epsilon> 0$ by continuity. Hence the equation $Q_1 = \epsilon^{n-2} (1 - \epsilon)$ can be solved for $\delta$, with the solution as given in (\ref{defofdeltap1pn}), if $\epsilon >0$ is sufficiently small. We thus obtain the claimed expansion
\begin{equation*}
h'' (\rho)= \frac{1}{\rho - \epsilon} + \hat{Q}_{0} + \hat{Q}_{1} (\rho - \epsilon) + \cdots  
\end{equation*}
near $\rho = \epsilon$, if we choose $\alpha$, $\beta$, $\gamma$, $\delta$ as in (\ref{defofalphap1pn}), (\ref{defofbetap1pn}), (\ref{defofgammap1pn}), (\ref{defofdeltap1pn}) and $\epsilon$ to be sufficiently small.

\end{proof}

Note that $Q_0 =0$ (resp. $Q_1 = \epsilon^{n-2} (1 - \epsilon)$) proved in the above is equivalent to saying $Q(\epsilon)=0$ (resp. $Q' (\epsilon) =  \epsilon^{n-2} (1 - \epsilon)$). Together with what was proved in Lemma \ref{rho1remsingab}, we summarise below the properties of the polynomial $Q(\rho)$ that we have established so far.

\begin{lemma} \label{lempropqp1pn}
For the choice of $\alpha$, $\beta$, $\gamma$, $\delta$ as in (\ref{defofalphap1pn}), (\ref{defofbetap1pn}), (\ref{defofgammap1pn}), (\ref{defofdeltap1pn}) and sufficiently small $\epsilon $, the polynomial $Q(\rho)$ satisfies the following properties:
\begin{enumerate}
\item $Q(1) = Q'(1)=0$, $Q''(1)=2$,
\item $Q (\epsilon) = 0$, $Q' (\epsilon) = \epsilon^{n-2} (1 -\epsilon)$.
\end{enumerate}
\end{lemma}
We also need the following estimates of $\alpha$, $\beta$, $\gamma$, $\delta$ in the later argument.
\begin{lemma} \label{lemestabcdine}
We can estimate $\alpha = O(\epsilon^2)$, $\beta = O(\epsilon^2)$, $\gamma =  O(\epsilon^2)$, $\delta = -n(n+1) +O(\epsilon^2)$, when $\epsilon$ is sufficiently small.
\end{lemma}


\begin{proof}
The proof is just a straightforward computation; we compute $\delta$ as
\begin{equation*}
\delta = \frac{\frac{2}{n(n-1)} + \frac{2(n+2)}{n^2 (n-1)} \epsilon + O(\epsilon^2)}{\frac{-2}{n^2(n-1)(n+1)} - \frac{2 (n+2)}{n^3 (n-1)(n+1)}\epsilon + O(\epsilon^2)} = -n (n+1) + O(\epsilon^2),
\end{equation*}
and similarly for $\gamma$. The claim for $\alpha$ and $\beta$ follows easily from the definitions (\ref{defofalphap1pn}) and (\ref{defofbetap1pn}).



\end{proof}


With these preparations, we now prove that $Q(\rho)$ is non-zero for all $\rho \in (\epsilon , 1)$.
\begin{lemma} \label{lemposqeps}
$Q(\rho) >0$ for all $\rho \in (\epsilon ,1)$, if $\alpha$, $\beta$, $\gamma$, $\delta$ are chosen as in (\ref{defofalphap1pn}), (\ref{defofbetap1pn}), (\ref{defofgammap1pn}), (\ref{defofdeltap1pn}), and $\epsilon$ is sufficiently small.
\end{lemma}

\begin{proof}
Note first of all that the second derivative of $Q$ can be computed as
\begin{equation*}
Q'' (\rho) = \rho^{n-3} \left[ - \gamma \rho^3 - (\delta - \gamma) \rho^2 -(n(n-1) -(2n+ \delta)) \rho +(n-1)(n-2) \right] .
\end{equation*}
Re-write the terms in the bracket $[ \cdots ]$ as
\begin{align*}
&- \gamma \rho^3 - (\delta - \gamma) \rho^2 -(n(n-1) -(2n+ \delta)) \rho +(n-1)(n-2) \\
&= \tilde{Q} (\rho) + \epsilon^2 \tilde{Q}_{\mathrm{rem}} (\rho) ,
\end{align*}
where we defined
\begin{equation*}
\tilde{Q} (\rho) := n(n+1) \rho^2 -2n (n-1 ) \rho + (n-1)(n-2)
\end{equation*}
and
\begin{equation*}
\tilde{Q}_{\mathrm{rem}} (\rho) := \frac{1}{\epsilon^2} \left( -\gamma \rho^3 - (\delta_0  - \gamma) \rho^2 - \delta_0 \rho \right)
\end{equation*}
with $\delta_0 := \delta + n(n+1)$. Recalling $\gamma = O(\epsilon^2)$ and $\delta_0 = O(\epsilon^2)$ (cf. Lemma \ref{lemestabcdine}), we see that there exists a constant $\tilde{C} (\epsilon_1) >0$, which depends only on (sufficiently small) $\epsilon_1$ and hence can be chosen uniformly for all $\epsilon$ satisfying $0 < \epsilon < \epsilon_1$, such that
\begin{equation} \label{estqtilderem}
|\tilde{Q}_{\mathrm{rem}} (\rho)| < \tilde{C} (\epsilon_1)
\end{equation}
holds for all $\rho \in (0, 1)$ and all $\epsilon$ satisfying $0 < \epsilon < \epsilon_1$.

Observe now that $\tilde{Q} (\frac{1}{2n}) = \frac{n+1}{4n} + (n-1)(n-3) >0$ for $n \ge 3$. Observe also that $\tilde{Q}' (\rho) = 2n (n+1) \rho -2n (n-1)$, meaning that $\tilde{Q} (\rho)$ is monotonically decreasing on $ (0, \frac{n-1}{n+1})$. Noting $\frac{1}{2n} < \frac{n-1}{n+1}$ if $n \ge 3$, we hence have
\begin{equation} \label{estqtildemdec}
\tilde{Q} (\rho) > \tilde{Q} \left( \frac{1}{2n} \right) >  \frac{n^2 (n-2) + 2n+1}{n}
\end{equation}
if $\rho \in (0, \frac{1}{2n})$. The estimates (\ref{estqtilderem}) and (\ref{estqtildemdec}) imply that, if $\epsilon$ is chosen to be sufficiently small,
\begin{equation*}
Q'' (\rho) = \rho^{n-3} \left( \tilde{Q} (\rho)+ \epsilon^2 \tilde{Q}_{\mathrm{rem}} (\rho)  \right)  >  \rho^{n-3} \left( \frac{n+1}{4n} + (n-1)(n-3) \right)>0
\end{equation*}
for all $\rho \in (0, \frac{1}{2n})$.

Now recall $Q(\epsilon) =0$ and $Q' (\epsilon) =\epsilon^{n-2} (1 - \epsilon) >0$ (cf. Lemma \ref{lempropqp1pn}). Since $Q'' (\rho)$ is strictly positive for all $\rho \in (0, \frac{1}{2n})$ if $\epsilon$ is chosen to be sufficiently small, $Q' (\rho)$ is strictly monotonically increasing on $ (0 , \frac{1}{2n})$. Combined with $Q' (\epsilon) = \epsilon^{n-2} (1 - \epsilon) >0$, we thus see that $Q' (\rho)>0$ for all $\rho \in (\epsilon, \frac{1}{2n})$. Thus $Q(\rho)$ is strictly monotonically increasing on $(\epsilon , \frac{1}{2n})$ if $\epsilon$ is chosen to be sufficiently small, but recalling $Q(\epsilon) =0$, we see that $Q(\rho)$ is strictly positive for all $\rho \in ( \epsilon , \frac{1}{2n})$ if $\epsilon$ is chosen to be sufficiently small.

Having established $Q(\rho) >0$ for all $\rho \in (\epsilon, \frac{1}{2n})$, we are now reduced to proving the positivity of $Q (\rho)$ for all $\rho \in [\frac{1}{2n} , 1)$ when $\epsilon$ is sufficiently small. We need some preparations (i.e. the estimate (\ref{estf1tildeeps0})) before doing so.

We now recall that $\delta_0 = \delta + n (n+1)$ is of order $\epsilon^2$ by Lemma \ref{lemestabcdine}, and write
\begin{equation} \label{qritmsofd01}
Q(\rho) = \rho^{n-1}( \rho -1 )^2 - \frac{\gamma }{(n+1)(n+2)} \rho^{n+2} - \frac{\delta_0 - \gamma }{n(n+1)} \rho^{n+1} + \frac{\delta_0}{n(n-1)}  \rho^n - \alpha - \beta \rho . 
\end{equation}
Note that, by Lemma \ref{lemestabcdine}, there exist real constants $\tilde{\alpha}$, $\tilde{\beta}$, $\tilde{\gamma}$, $\tilde{\delta}$ (when $\epsilon$ is sufficiently small) which remain bounded as $\epsilon \to 0$ such that $\alpha = \tilde{\alpha} \epsilon^2$, $\beta = \tilde{\beta} \epsilon^2$, $\gamma = \tilde{\gamma} \epsilon^2$, $\delta_0 = \tilde{\delta} \epsilon^2$. We can thus write
\begin{equation*}
Q(\rho) = \rho^{n-1}( \rho -1 )^2 - \epsilon^2 \left( \frac{\tilde{\gamma} }{(n+1)(n+2)} \rho^{n+2} + \frac{\tilde{\delta} - \tilde{\gamma} }{n(n+1)} \rho^{n+1} - \frac{\tilde{\delta}}{n(n-1)}  \rho^n + \tilde{\alpha} + \tilde{\beta} \rho \right) .
\end{equation*}
Suppose that we write
\begin{equation*}
\tilde{F}_0 (\rho) := \frac{\tilde{\gamma} }{(n+1)(n+2)} \rho^{n+2}+ \frac{\tilde{\delta} + \tilde{\gamma} }{n(n+1)} \rho^{n+1} - \frac{\tilde{\delta}}{n(n-1)}  \rho^n + \tilde{\alpha} + \tilde{\beta} \rho 
\end{equation*}
for the terms in the bracket. Now recall that $Q(\rho)$ has a zero of order exactly 2 at $\rho =1$ by Lemma \ref{lempropqp1pn}. This means that $\tilde{F}_0$ must have a zero of order at least 2 at $\rho =1$, and hence we can factorise 
\begin{equation*} 
\tilde{F}_0 (\rho) = (\rho-1)^2 \tilde{F}_1 (\rho)
\end{equation*}
for some polynomial $\tilde{F}_1 (\rho)$. Observe that this implies
\begin{equation} \label{defofftildep1pn}
Q(\rho) = (\rho -1)^2 (\rho^{n-1} - \epsilon^2 \tilde{F}_1 (\rho)) .
\end{equation}
Note that, since $\tilde{\alpha}$, $\tilde{\beta}$, $\tilde{\gamma}$, $\tilde{\delta}$ are uniformly bounded for all sufficiently small $\epsilon>0$, there exists a constant $\tilde{C}_1 (\epsilon_1) >0$, which depends only on (sufficiently small) $\epsilon_1$ and hence can be chosen uniformly for all $\epsilon$ satisfying $0 < \epsilon < \epsilon_1$, such that
\begin{equation} \label{estf1tildeeps0}
| \tilde{F}_1 (\rho) |< \tilde{C}_1 (\epsilon_1)
\end{equation}
holds for all $\rho \in (0 ,1)$ and all $\epsilon$ satisfying $0 < \epsilon < \epsilon_1$.

Now consider the equation (\ref{defofftildep1pn}) for $\rho \in [\frac{1}{2n},1)$. Suppose $Q(\rho_0) =0$ at some $\rho_0 \in [\frac{1}{2n} ,1)$. We would then have $\rho_0^{n-1} - \epsilon^2 \tilde{F}_1 (\rho_0) =0$. However, since $ \rho_0^{n-1} \ge (2n)^{-n+1}$ and $\tilde{F}_1$ is uniformly bounded on $ [\frac{1}{2n} ,1)$ (as given in (\ref{estf1tildeeps0})), we have $\epsilon^2  \tilde{F}_1 \to 0$ uniformly on $ [\frac{1}{2n} ,1)$ as $\epsilon \to 0$, and hence the equation $\rho_0^{n-1} - \epsilon^2 \tilde{F}_1 (\rho_0) =0$ cannot hold if we take $\epsilon$ to be sufficiently small. We thus get $Q(\rho) \neq 0$ for all $\rho \in [\frac{1}{2n} ,1)$. Since $Q(\rho) >0$ on $(\epsilon , \frac{1}{2n})$, we get $Q(\rho) >0$ for all $\rho \in [\frac{1}{2n} ,1)$ by continuity, and finally establish $Q(\rho) >0$ for all $\rho \in (\epsilon , 1)$ and all sufficiently small $\epsilon >0$.

\end{proof}


We shall finally prove the positive-definiteness of the Hessian of the symplectic potential
\begin{equation}
s(x) = \frac{1}{2} \left( \sum_{i=1}^n x_i \log x_i + (1-r) \log (1-r) + h(\rho)   \right) .
\end{equation}
Writing $s^{FS}_{ij}$ for the Hessian of the symplectic potential corresponding to the Fubini--Study metric on $\mathbb{P}^n$, i.e.
\begin{align} 
s^{FS}_{ij} 
&:= \frac{1}{2}  \frac{\partial^2}{\partial x_i \partial x_j} \left( \sum_{i=1}^n x_i \log x_i + (1-r) \log (1-r) \right) \notag \\
&=\frac{1}{2} 
\begin{pmatrix}
x_1^{-1} & 0 & 0 & \hdots & 0 \\
0 & x_2^{-1} & 0 & \hdots & 0 \\
\vdots & \vdots & \vdots & \ddots & \vdots \\
0 & 0 & 0 & \hdots & x_n^{-1}
\end{pmatrix}
+ \frac{1}{2}  \frac{1}{1-r}
\begin{pmatrix}
1 & 1 & \cdots & 1 & 1 \\ 1 & 1 & \cdots & 1 & 1 \\ \vdots & \vdots & \ddots & \vdots & \vdots \\ 1 & 1 & \cdots & 1 & 1 
\end{pmatrix} , \label{sfssymppotent}
\end{align}
we can write the Hessian $s_{ij}$ of $s$ as
\begin{equation*}
s_{ij} = s^{FS}_{ij} + \frac{h''}{2}  T_{ij} ,
\end{equation*}
where $T$ is a matrix defined by
\begin{equation*}
T:=
\begin{pmatrix}
 1 & \cdots & 1 & 0  \\  \vdots & \ddots & \vdots & \vdots \\  1 & \cdots & 1 & 0 \\  0& \cdots & 0 & 0
\end{pmatrix} .
\end{equation*}
Observe that $T$ is positive semi-definite.

\begin{lemma} \label{lemsijposdef}
$s_{ij}$ is positive definite on the interior $\mathcal{P}^{\circ}_{\epsilon} (X)$ of the polytope $\mathcal{P}_{\epsilon} (X)$ and has the determinant of the form (\ref{abrcxthmhesssdet}), if $\alpha$, $\beta$, $\gamma$, $\delta$ are chosen as in (\ref{defofalphap1pn}), (\ref{defofbetap1pn}), (\ref{defofgammap1pn}), (\ref{defofdeltap1pn}), and $\epsilon$ is sufficiently small.
\end{lemma}

\begin{proof}
Observe first of all that, since $s^{FS}_{ij}$ is positive definite (as given in (\ref{sfssymppotent})) and $T$ is positive semi-definite, it suffices to prove that there exists a constant $C (\epsilon_1)>0$, which depends only on some (small) $\epsilon_1 >0$ and hence can be chosen uniformly for all $\epsilon$ satisfying $0 < \epsilon < \epsilon_1$, such that 
\begin{equation} \label{hdprhogtmepsc}
h'' (\rho) > - \epsilon C (\epsilon_1)
\end{equation}
holds for all $\rho \in ( \epsilon , 1) $ and all $\epsilon$ satisfying $0 < \epsilon < \epsilon_1$; the claimed positive-definiteness would then follow by taking $\epsilon$ to be sufficiently small.

The inequality (\ref{hdprhogtmepsc}) also implies that $\det (s_{ij})$ is of the form required in (\ref{abrcxthmhesssdet}); by a straightforward computation, representing $s_{ij}$ with respect to the following basis
\begin{equation*}
e_1 =
\begin{pmatrix}
1 \\
 1 \\
  \vdots \\
 1 \\
 0
\end{pmatrix} 
e_2 = 
\begin{pmatrix}
x_1^{-1} \\
 -x_2^{-1} \\
 0 \\
  \vdots \\
  0
\end{pmatrix} ,
\dots , 
e_{n-1} = 
\begin{pmatrix}
x_1^{-1} \\
0 \\
  \vdots \\
   -x_{n-1}^{-1} \\
  0
\end{pmatrix} ,
e_n = 
\begin{pmatrix}
0 \\
 0 \\
  \vdots \\
 0 \\
 1
\end{pmatrix} ,
\end{equation*}
we see that
\begin{align*}
\det (s_{ij}) &= 
\frac{1}{2^n} \prod_{i=1}^n x_i^{-1} \det
\begin{pmatrix}
1 +h'' + \frac{\rho}{1-r}  & 0 & \cdots & 0 & \frac{x_n}{1-r} \\
0 & 1 & \cdots & 0 & 0 \\
\vdots & \vdots & \ddots & \vdots & \vdots \\
0 & 0 &  \cdots & 1 & 0 \\
\frac{\rho}{1-r} & 0 & \cdots & 0 & 1 + \frac{x_n}{1-r}
\end{pmatrix} \\
&=\frac{1}{2^n} \prod_{i=1}^n x_i^{-1}  \frac{1}{1-r} \left( (1 + h'')(1 - \rho)+ \rho \right).
\end{align*}
Granted (\ref{hdprhogtmepsc}), we thus see that $\det (s_{ij})$ is of the form required in (\ref{abrcxthmhesssdet}), by taking $\epsilon>0$ to be sufficiently small and also by recalling Lemmas \ref{rho1remsingab}, \ref{lempseofhdp}, and \ref{lemposqeps}.

We now prove (\ref{hdprhogtmepsc}). Throughout in the proof, $C (\epsilon_1)$ will denote a constant which depends only on $\epsilon_1$ (and not on $\epsilon$) which varies from line to line.


Now define
\begin{equation*}
\tilde{F}_2 (\rho) := - \frac{\gamma }{(n+1)(n+2)} \rho^{n+2} - \frac{\delta_0 - \gamma }{n(n+1)} \rho^{n+1} + \frac{\delta_0}{n(n-1)}  \rho^n - \alpha - \beta \rho
\end{equation*}
so that
\begin{equation*}
Q(\rho) =  \rho^{n-1}(1 -  \rho  )^2 + \tilde{F}_2 (\rho) .
\end{equation*}


Observe first that $Q(\epsilon) =0$ (cf. Lemma \ref{lempropqp1pn}) is equivalent to $\tilde{F}_2 (\epsilon) = -\epsilon^{n-1} (1 - \epsilon)^2$. On the other hand, $\gamma = \delta_0 = O(\epsilon^2)$ (cf. Lemma \ref{lemestabcdine}) implies $\tilde{F}_2 (\epsilon) = O(\epsilon^{n+2}) - \alpha - \beta \epsilon$. Thus we get
\begin{equation*}
 - \alpha - \beta \epsilon = - \epsilon^{n-1} (1 - \epsilon)^2 + O(\epsilon^{n+2}), 
\end{equation*}
and hence
\begin{equation*}
- \alpha - \beta \rho = - \alpha - \beta \epsilon - \beta (\rho - \epsilon) = - \epsilon^{n-1} (1 - \epsilon)^2 + O(\epsilon^{n+2}) - \beta (\rho - \epsilon).
\end{equation*}
On the other hand, since $Q ' (\epsilon) = \epsilon^{n-2} (1 - \epsilon)$ (cf. Lemma \ref{lempropqp1pn}), we have
\begin{equation*}
Q '(\epsilon) = \epsilon^{n-2} (1 - \epsilon) [ (n-1) (1 - \epsilon) - 2 \epsilon] - \beta + O(\epsilon^{n+1}) = \epsilon^{n-2} (1 - \epsilon),
\end{equation*}
by differentiating (\ref{qritmsofd01}) and recalling Lemma \ref{lemestabcdine}. We thus get
\begin{align*}
\beta &=-  \epsilon^{n-2} (1 - \epsilon) + \epsilon^{n-2} (1 - \epsilon) [(n-1) (1 - \epsilon) - 2 \epsilon]  + O(\epsilon^{n+1}) \\
&= \epsilon^{n-2} (1 - \epsilon) [(n-1) (1 - \epsilon) - 1- 2 \epsilon]  + O(\epsilon^{n+1}).
\end{align*}
Define a constant
\begin{equation*}
\bar{C}_{\epsilon} := (n-1) (1 - \epsilon) - 1- 2 \epsilon
\end{equation*}
and observe that it satisfies the following bound
\begin{equation} \label{estofcbarepsn2}
\frac{10n^2 - 21n -1}{10n} \le \bar{C}_{\epsilon}  < n-2 
\end{equation}
for all $0< \epsilon< \frac{1}{10n}$ say, where we note $10n^2 - 21n -1 >0$ if $n \ge 3$; $\bar{C}_{\epsilon}$ can be bounded from above and below by a positive constant, uniformly of (all small enough) $\epsilon$. Then we can write $\beta = \bar{C}_{\epsilon}  \epsilon^{n-2} (1 - \epsilon) + O(\epsilon^{n+1})$, and hence
\begin{align*}
- \alpha - \beta \rho  &= - \epsilon^{n-1} (1 - \epsilon)^2 + O(\epsilon^{n+2}) - \beta (\rho - \epsilon) \\
&= - \epsilon^{n-2} (1 - \epsilon) [ \epsilon (1-  \epsilon) + (\bar{C}_{\epsilon}+O(\epsilon^3))  (\rho - \epsilon) + O(\epsilon^4)] .
\end{align*}
We now write
\begin{equation*}
\frac{\rho^{n-1} (1 - \rho)^2}{Q(\rho)} -1 = \frac{1}{1+\epsilon^2 \tilde{F}_3 (\rho)  - \tilde{F}_4 (\rho) } - 1
\end{equation*}
where we defined
\begin{equation*}
\tilde{F}_3 (\rho) := \frac{1}{\epsilon^2 (1 -\rho)^2} \left( - \frac{\gamma }{(n+1)(n+2)} \rho^{3} - \frac{\delta_0 - \gamma }{n(n+1)} \rho^{2} + \frac{\delta_0}{n(n-1)}  \rho \right)
\end{equation*}
and
\begin{equation*}
\tilde{F}_4 (\rho) :=  \frac{(\epsilon / \rho)^{n-2}  (1 - \epsilon)}{(1 - \rho)^2} [ (\epsilon / \rho) (1-  \epsilon +O(\epsilon^3)) + (\bar{C}_{\epsilon} +O(\epsilon^3)) (1 - \epsilon / \rho)  ] .
\end{equation*}

Arguing as we did in (\ref{estqtilderem}) and (\ref{estf1tildeeps0}), we use $\delta_0 = O(\epsilon^2)$ and $\gamma = O(\epsilon^2)$ (cf. Lemma \ref{lemestabcdine}) to see that $\tilde{F}_3 (\rho)$ satisfies
\begin{equation} \label{eststildeceps0}
|\tilde{F}_3 (\rho)| < C (\epsilon_1)
\end{equation}
for all $\rho \in (0 , \frac{1}{2})$ say, with a constant $C (\epsilon_1) >0$ which depends only on (sufficiently small) $\epsilon_1$ and hence can be chosen uniformly for all $\epsilon$ satisfying $0 < \epsilon < \epsilon_1$. Note also that the estimate (\ref{estofcbarepsn2}) implies
\begin{equation*}
\tilde{F}_4 (\rho) =   \frac{(\epsilon / \rho)^{n-2}  (1 - \epsilon)}{(1 - \rho)^2} [ (\epsilon / \rho) (1-  \epsilon +O(\epsilon^3)) + (\bar{C}_{\epsilon} +O(\epsilon^3)) (1 - \epsilon / \rho)  ]  >0
\end{equation*}
for all $\rho \in (\epsilon , 1 )$ if $\epsilon$ is small enough. Finally, observe that
\begin{equation*}
\frac{\rho^{n-1} (1 - \rho)^2}{Q(\rho)} = \frac{1}{1+\epsilon^2  \tilde{F}_3 (\rho)  -  \tilde{F}_4 (\rho)} 
\end{equation*}
and that $Q (\rho )>0$ for $\rho \in ( \epsilon ,1 )$ (cf. Lemma \ref{lemposqeps}) imply $ 1 + \epsilon^2 \tilde{F}_3 (\rho) - \tilde{F}_4 (\rho)>0$ for all $\rho \in (\epsilon, 1 )$. We thus have
\begin{equation*}
0<  \tilde{F}_4 (\rho) < 1 + \epsilon^2 \tilde{F}_3 (\rho)
\end{equation*}
for all $\rho \in (\epsilon , 1)$ if $\epsilon$ is small enough.

Hence we have
\begin{align*}
 \frac{\rho^{n-1} (1 - \rho)^2}{Q(\rho)} -1  &= \frac{1}{1+\epsilon^2  \tilde{F}_3 (\rho)  -  \tilde{F}_4 (\rho)} - 1 \\
 & >  \frac{1}{1+\epsilon^2  \tilde{F}_3 (\rho) } - 1 ,
\end{align*}
for all $\rho \in (\epsilon , 1 )$. In particular, recalling the estimate (\ref{eststildeceps0}), there exists a constant $C (\epsilon_1)>0$ independent of $\epsilon$ such that
\begin{align*}
h'' (\rho) &= \frac{1}{\rho} \frac{1}{1 - \rho} \left(  \frac{\rho^{n-1} (1 - \rho)^2}{Q(\rho)} -1 \right) \\
&> \frac{1}{\rho} \frac{1}{1 - \rho} \left( \frac{1}{1+\epsilon^2  \tilde{F}_3 (\rho) } - 1 \right) \\
&>- \frac{\epsilon}{2}  \left| \tilde{F}_3 (\rho)  \right| \left( 1+ \epsilon^2 \left|  \tilde{F}_3 (\rho) \right| +  \cdots \right) \\
&> - \epsilon C (\epsilon_1)
\end{align*}
for all $\rho \in (\epsilon, \frac{1}{2})$, if $\epsilon$ is chosen to be sufficiently small.

Having established the claim for all $\rho \in (\epsilon , \frac{1}{2})$, we now treat the case $\rho \in [\frac{1}{2} ,1)$. Using the polynomial $\tilde{F}_1$ as given in (\ref{defofftildep1pn}), we can write
\begin{equation*}
h'' (\rho)
=  \frac{1}{\rho}\frac{1}{1- \rho} \left(\frac{\rho^{n-1} }{ \rho^{n-1} - \epsilon^2 \tilde{F}_1 (\rho) } -  1 \right) .
\end{equation*}
We thus find that we have a power series expansion of $h''$ in $\epsilon^2 \rho^{-n+1} \tilde{F}_1 $ as
\begin{equation*}
h'' (\rho) =  \frac{1}{\rho}\frac{1}{1- \rho} \left(\frac{1 }{ 1 - \epsilon^2 \rho^{-n+1} \tilde{F}_1 (\rho) } -  1 \right) = \epsilon^{2}  \frac{\rho^{-n} \tilde{F}_1 (\rho)}{1- \rho} \left(1+  \epsilon^{2} \rho ^{-n+1} \tilde{F}_1 (\rho) + \cdots  \right)
\end{equation*}
where the series in the bracket is uniformly convergent on $ [ \frac{1}{2} ,1)$ for all $0 < \epsilon < \epsilon_1 $ if $\epsilon_1$ is chosen to be sufficiently small, by noting
\begin{equation*}
|\rho^{-n+1} \tilde{F}_1 (\rho)|  < C (\epsilon_1)
\end{equation*}
for $\rho \in [ \frac{1}{2} ,1)$, following from the estimate (\ref{estf1tildeeps0}). We thus find
\begin{equation*}
|h'' (\rho)| < \epsilon^{2} C(\epsilon_1) \left| \frac{\rho^{-n} \tilde{F}_1 (\rho)}{1- \rho} \right| ,
\end{equation*}
for a constant $C(\epsilon_1)>0$ which does not depend on $\epsilon$. Recall now that $Q'' (1) =2$ (cf. Lemma \ref{lempropqp1pn}) and $Q(\rho) = (\rho -1)^2 ( \rho^{n-1}  - \epsilon^2 \tilde{F}_1 (\rho) )$ (cf. equation (\ref{defofftildep1pn})) imply $\tilde{F}_1 (1) =0$. We thus see that $\frac{\tilde{F}_1 (\rho) }{1 - \rho}$ is in fact a polynomial, and hence by arguing as we did in (\ref{estqtilderem}) and (\ref{estf1tildeeps0}), we get
\begin{equation*}
\left| \frac{\rho^{-n} \tilde{F}_1 (\rho)}{ 1- \rho} \right| < C(\epsilon_1)
\end{equation*}
uniformly on $ [ \frac{1}{2} ,1)$ and for all $\epsilon$ satisfying $0< \epsilon < \epsilon_1$, if $\epsilon_1$ is sufficiently small. We can thus evaluate $|h'' (\rho)| < \epsilon^{2}  C(\epsilon_1)$ for all $\rho \in [ \frac{1}{2} ,1)$, which finally establishes $h'' (\rho) > -\epsilon {C} (\epsilon_1)$ for all $\rho \in (\epsilon ,1 )$ and all $\epsilon$ satisfying $0 < \epsilon < \epsilon_1$, where $\epsilon_1$ is chosen to be sufficiently small.


\end{proof}

\subsection{Potential extension of Proposition \ref{mainthmextp1pn}}

As we saw in the above, the hypothesis $\epsilon \ll 1$ is essential in establishing the regularity (Proposition \ref{reghmaintr}) of the symplectic potential. However, as in the point blowup case (Theorem \ref{calextpclbpn}), it is natural to expect that the extremal metrics exist in each \kah class.
\begin{question}
Does Proposition \ref{reghmaintr} hold for any $0 < \epsilon <1$? In other words, does $\mathrm{Bl}_{\prj^1} \prj^n$ admit an extremal metric in each \kah class?
\end{question}
Some numerical results obtained by a computer experiment seem to suggest that the answer to this question should be affirmative.

\subsection*{Acknowledgements}
Most of the work presented in this paper was carried out at the University of Edinburgh under the supervision of Michael Singer, and the results in this paper originally appeared, with a sketch proof, in the author's first year report submitted to the University of Edinburgh in 2012. This work also forms part of the author's PhD thesis submitted to the University College London. He thanks both universities for supporting his studies, and Michael Singer for suggesting Problem \ref{blupprbhdmfd} and teaching him several facts on $\mathrm{Bl}_{\prj^1} \prj^n$ that are used in \S \ref{premointth}. He is grateful to Ruadha\'i Dervan, Joel Fine, Jason Lotay, and Julius Ross for helpful comments that improved this paper. 




\bibliographystyle{amsplain}
\bibliography{2015_22_ThesisTotal}

\begin{flushleft}
DEPARTMENT OF MATHEMATICS, UNIVERSITY COLLEGE LONDON \\
Email: \verb|yoshinori.hashimoto.12@ucl.ac.uk|, \verb|yh292@cantab.net|
\end{flushleft}

\end{document}